\documentclass[11pt]{article}
\usepackage{bbm}
\usepackage{authblk}
\usepackage{url}
\usepackage{tikz}
\usepackage{pdflscape}
\usepackage{longtable}
\usepackage{booktabs}
\usepackage{colortbl}
\usepackage{verbatim}
\usepackage{amssymb}
\usepackage{amsmath}
\usepackage{bm}
\usepackage{amsthm, color}
\usepackage{fullpage}
\usepackage{etoolbox}
\allowdisplaybreaks

\newcommand{\C}{\mathbb{C}}

\newcommand{\R}{\mathbb{R}}

\newcommand{\de}{\partial}

\DeclareMathOperator{\curl}{curl}
\DeclareMathOperator{\divv}{div}

\newtheorem{thm}{Theorem}[section]
\newtheorem{cor}[thm]{Corollary}

\newtheorem{claim}[thm]{Claim}
\newtheorem{lem}[thm]{Lemma}
\newtheorem{prop}[thm]{Proposition}
\newtheorem{defi}[thm]{Definition}
\newtheorem{rmk}[thm]{Remark}

\newtheorem{assumption}[thm]{Assumption}

\newtheorem{propos}{Proposition}[section]

\makeatletter
\DeclareRobustCommand\widecheck[1]{{\mathpalette\@widecheck{#1}}}
\def\@widecheck#1#2{%
    \setbox\z@\hbox{\m@th$#1#2$}%
    \setbox\tw@\hbox{\m@th$#1%
       \widehat{%
          \vrule\@width\z@\@height\ht\z@
          \vrule\@height\z@\@width\wd\z@}$}%
    \dp\tw@-\ht\z@
    \@tempdima\ht\z@ \advance\@tempdima2\ht\tw@ \divide\@tempdima\thr@@
    \setbox\tw@\hbox{%
       \raise\@tempdima\hbox{\scalebox{1}[-1]{\lower\@tempdima\box
\tw@}}}%
    {\ooalign{\box\tw@ \cr \box\z@}}}
\makeatother

\numberwithin{equation}{section}

\begin{document}

\title{\sc Operator-norm homogenisation estimates for the system of Maxwell equations on periodic singular structures}

\author[1]{Kirill Cherednichenko}
\author[1]{Serena D'Onofrio}
\affil[1]{Department of Mathematical Sciences, University of Bath, Claverton Down, Bath, BA2 7AY, United Kingdom}




\maketitle

\begin{abstract}
For arbitrarily small values of $\varepsilon>0,$ we formulate and analyse the Maxwell system of equations of electromagnetism on $\varepsilon$-periodic sets
$S^\varepsilon\subset{\mathbb R}^3.$ Assuming that a family of Borel measures $\mu^\varepsilon,$ such that 
${\rm supp}(\mu^\varepsilon)=S^\varepsilon,$ is obtained by $\varepsilon$-contraction of a fixed 1-periodic measure $\mu,$ and for right-hand sides $f^\varepsilon\in L^2(\R^3, d\mu^\varepsilon),$ we prove order-sharp norm-resolvent convergence estimates for the solutions of the system.
Our analysis includes the case of periodic ``singular structures", when $\mu$ is supported 
by lower-dimensional manifolds. The estimates are obtained by combining several new tools we develop for analysing the Floquet decomposition of an elliptic differential operator on functions from Sobolev spaces with respect to a periodic Borel measure.  These tools include a generalisation of the classical Helmholtz decomposition for $L^2$ functions, an associated Poincar\'{e}-type inequality, uniform with respect to the parameter of the Floquet decomposition, and an appropriate asymptotic expansion inspired by the classical power series. Our technique does not involve any spectral analysis and does not rely on the existing approaches, such as Bloch wave homogenisation or the spectral germ method.

\vskip 0.3cm


\end{abstract}

\section{Introduction}

The operator-theoretic perspective on partial differential equations (PDE) with multiple scales has proved effective for obtaining sharp convergence results for problems of periodic homogenisation, see {\it e.g.} \cite{Sevostianova, Zhi89, BS04, BirmanSuslina_corrector, ZhikovPastukhova, CC16, Suslina_dyrki} for related developments in the ``whole-space'' setting, 
{\it i.e.} when the spatial domain is periodic is invariant with respect to shifts by the elements of a periodic lattice in ${\mathbb R}^d,$ $d\ge 2.$

The techniques developed in the above works have highlighted a variety of different new ways to interpret the process homogenisation, {\it e.g.} via the singular-value decomposition of operator resolvents or by extending the classical perturbation series to PDE families that involve an additional length-scale parameter. However, a common strand in all of them is the idea that homogenisation corresponds is a ``long-wave" asymptotic regime \cite[Chapter 4]{BLP}, governed by the behaviour of the related differential operators near the bottom of its spectrum. It seems natural to enquire whether this rationale can be extended to arbitrary periodic (Borel) measures, providing useful order-sharp approximations for periodic ``structures".  

In our earlier work \cite{CD18} we addressed the above question for the case of a scalar elliptic problem
\begin{equation}
-\nabla\cdot A(\cdot/\varepsilon)\nabla u+u=f,\qquad f\in L^2({\mathbb R}^d,d\mu^\varepsilon),\qquad \varepsilon>0,
\label{whole_space_eq1}
\end{equation}
where the $\varepsilon$-periodic measure $\mu^\varepsilon$ is obtained by $\varepsilon$-scaling from a fixed periodic measure in ${\mathbb R}^d,$
and the matrix-function $A$ is uniformly positive definite. As a starting point of our approach,
we considered the PDE family obtained from (\ref{whole_space_eq1}) by the Floquet transform (see \cite{CC16}, \cite{ZP16}), in some sense replacing the macroscopic variable by an additional parameter $\theta$ (``quasimomentum"), akin to the Fourier dual variable for PDE with constant coefficients. The strategy for the analysis of the family thus obtained was to use an asymptotic approximation for the solution in powers of $\varepsilon,$ carefully analyse the homogenisation corrector as a function of $\varepsilon$ and $\theta,$ and obtain an estimate for the remainder that is uniform with respect to $\theta.$ The key technical tool for the proof of remainder estimates was a Poincar\'{e}-type inequality in an appropriate Sobolev space of quasiperiodic functions, conditioned by the fact that we deal with an arbitrary measure. Equipped with this new machinery, in the present paper we set out to tackle a vector problem, in particular the system of Maxwell equations, which is of interest in applications to electromagnetism. In the case when $\mu^\varepsilon$ is the Lebesgue measure, operator-norm estimates for the Maxwell system have been obtained, using the spectral approach, in \cite{BS04} (for the ``non-magnetic" case with no currents, as an application of the spectral germ technique introduced in the paper), \cite{Sus05} (for magnetic field and induction in the presence of currents), \cite{BS07} (for the full system in the non-magnetic case), and \cite{Sus08} (for the general Lebesgue measure case).

A research programme similar to the above, although outside the context of thin structures and using a different analytical approach, has been pursued by Birman, Suslina, and subsequently by Suslina and her students, starting with \cite{BS04, BirmanSuslina_corrector}. At the heart of their technique is the notion of a spectral germ for a class of operator pencils, which quantifies the leading order of frequency dispersion of waves in a heterogeneous medium near the bottom of the spectrum of the associated differential operator with periodic coefficients. Complemented with the study of a Cauchy integral for a suitable operator-valued function of the spectral parameter, the analysis of the spectral germ allows one to obtain sharp operator-norm estimates for the resolvents in the direct integral representing the original operator via the standard Floquet-Bloch-Gelfand decomposition (parametrised by the quasimomentum $\chi,$ as we mention above), see {\it e.g.} \cite{Kuchment} for the background on Floquet theory.

{

Before proceeding to an extended summary of our results, we give an overview of existing literature on homogenisation methods that uses a version of the Floquet transform as a starting point. We should emphasise that this is the only crossover point of the approach of the present paper with this existing work: we do not use the spectral method, and the centrepieces of our analysis are a Helmholtz-type decomposition for vector fields, see Section \ref{Helmholtz_section} and a related Poincar\'{e}-type inequality (Assumption \ref{ass1} below), which we postulate and show to hold for some specific classes of singular measures, see Appendix B. Of course, the Poincar\'{e} inequality by itself provides information on the spectra of the operators involved, however we do not pursue this link and do not require it for the proof of the operator-norm convergence estimates. Instead, we develop a new tool for 
 proving the estimates, namely asymptotic expansions that are uniform in the quasimomentum, see Section \ref{proof_sec}. These are particularly effective in addressing the currently open problem of obtaining operator-norm estimates for the full system of Maxwell equations, the subject of our forthcoming paper \cite{CD20}.    

The existing works based on the applying the Floquet transform (equivalently, Bloch transform, Gelfand transform) to the original PDE 
are focussed around two cognate approaches to homogenisation, 
namely the so-called Bloch-wave homogenisation method and the operator germ technique mentioned above. Both stem from the idea 
that the macroscopic behaviour of PDE with periodic rapidly oscillating coefficients is related to the behaviour of the associated operator near the bottom of its spectrum --- an idea that goes back to  \cite{BLP} in the mathematical literature and some 30 years earlier \cite{Brillouin} in physics --- complemented with an appropriate perturbation analysis aimed at obtaining convergence estimates. The Bloch wave method has been more popular in the applied analysis community, aiming at the derivation of asymptotic models for heterogeneous media, however yielding weaker convergence statements than the analysis of the 
spectral projections of the operators entering the direct fibre decomposition and the study of an associated ``spectral germ" at the left edge of the spectrum. It can be argued, however, that from the point of view of quantitative error control in applications, the operator-norm analysis is preferable. This is especially important for the development of new tools for tackling problems in the currently intensive area of metamaterials, where resonant behaviour on the microscale necessitates operator-norm analysis, while formal approaches yield results for which only strong (and in some cases only weak) convergence can be established rigorously, see e.g. \cite{Wellander1, Wellander2, Sjoeberg_et_al}.
 
Among papers on Bloch-wave homogenisation, we should mention several works that have set a foundation for the method and established strong resolvent convergence in the classical setting of a scalar elliptic second-order PDE \cite{Allaire_Conca_1997, AC_Comptes_Rendus, ACV, CV, COV, CV_2002, SGV_2004}, leading to the analysis of the high-frequency spectrum, still in the framework of strong convergence \cite{AC}. Subsequently, the approach was applied to address 
the Stokes equation \cite{Allaire_et_al_Stokes_2007, AGV}, the heat equation \cite{BFM, OZ} and the above-cited earlier work \cite{Zhi89},
the system of equations of elasticity \cite{SGV_2005}, and fluid-solid interactions \cite{Conca_Lund, AC_fluid_solid, ACP}. In the context of hyperbolic problems, the Bloch wave method naturally leads to dispersive effective equations \cite{SS, DLS, Lamacz, ABV}, by picking up higher-order terms in the Bloch wave expansion. On the analytic side, the Bloch wave approach has been developed in the direction of the treatment of bounded domains \cite{COV_bounded}, the analysis of a class  formulations in terms of arbitrary Borel measures \cite{BF, BFR}, and bounds on effective properties \cite{BBF}.


Simultaneously, the spectral germ method, initiated by \cite{BS04}, has proved fruitful in obtaining operator-norm and energy estimates for a number of related problems:  boundary-value operators \cite{Suslina_Dirichlet, Suslina_Neumann}, parabolic semigroups \cite{Suslina_parabolic, Suslina_parabolic_corrector, Meshkova_Suslina}, hyperbolic groups \cite{BirmanSuslina_hyperbolic, Meshkova_hyperbolic_Math_Notes, BirmanSuslina_hyperbolic, Meshkova_hyperbolic_full, Meshkova_hyperbolic_Math_Notes}, perforated domains \cite{Suslina_dyrki}. 
Two further technical milestones for the progress along this avenue are boundary-layer analysis for bounded domains (as in \cite{Suslina_Dirichlet, Suslina_Neumann}) and two-parametric operator-norm estimates \cite{Suslina_two_parametric}.  It seems natural to conjecture that similar developments could be pursued in the context of arbitrary Borel measures, using the technique of the present paper, which we postpone to future publications.

An overview of the existing approaches to obtaining operator-norm estimates would not be complete without mentioning also the works \cite{Gri04, Griso_2006, Kenig} that use, respectively, the method of periodic unfolding and the analysis of boundary integral representations, as well as the paper \cite{ZhikovPastukhova} cited above, based on the analysis of the homogenisation corrector via ``Steklov smoothing", and the recent papers \cite{CEK19, CEK20}, which employ an analysis of appropriate Dirichlet-to-Neumann (or Poincar\'{e}-Steklov) operators.
 The methods of these works could also be considered in the context of thin and singular structures, however we refrain from pursuing the related discussion here.

We now turn to the description of our approach and its application to the system of Maxwell equations. Before proceeding to a more detailed description of the problem setup, we note that some ideas of the present paper (uniform Poincar\'{e} inequality and uniform asymptotic expansions for the fibre operators) have been implemented for norm-resolvent analysis of the behaviour of thin plates in the context of three-dimensional linearised elasticity \cite{CherVel}, and that the benefit of operator-norm estimates for quantitative analysis of the full time-dependent system of Maxwell equations has been recently demonstrated in \cite{DS} (albeit under the assumption of constant permeability of the medium). 
}

Consider a $Q$-periodic Borel measure $\mu$ on $\R^3$, where $Q=[0,1)^3$, such that $\mu(Q)=1$. For each $\varepsilon>0$ we define the ``$\varepsilon$-scaling" of $\mu,$ {\it i.e.} the $\varepsilon$-periodic measure $\mu^\varepsilon$ given by $\mu^\varepsilon(B)=\varepsilon^3\mu(\varepsilon^{-1}B)$ for all Borel sets $B\subset \R^3,$ so that $\mu^1\equiv\mu.$ Henceforth, we denote by $C_0^\infty ({\mathbb R}^3)$ the space of infinitely smooth functions with compact support on 
${\mathbb R}^3$ and by $L^2(\R^3,d\mu^\varepsilon)$ the space of functions with values in $\C^3$ that are square integrable over $\R^3$ with respect to the measure $\mu^\varepsilon.$ Throughout the paper, for vectors $a, b\in \C^3$ we denote by $a\cdot b$ their standard (sesquilinear) Euclidean inner product, and define all function spaces over the field ${\mathbb C}.$ 

We aim at analysing the long-scale properties of periodic structures described by the measures $\mu^\varepsilon,$ in the context of the Maxwell system of equations of electromagnetism, see {\it e.g.} \cite{Jackson}, \cite{Cessenat}. More precisely, in what follows we study the asymptotic behaviour, as $\varepsilon\to 0$, of the solutions $u^\varepsilon$ to the vector problems
\begin{equation}
\label{curleq_u}
\curl \bigl(A(\cdot / \varepsilon) \curl u^\varepsilon\bigr)+u^\varepsilon=f^\varepsilon, \quad \quad f^\varepsilon\in L^2(\R^3, d\mu^\varepsilon),
\end{equation}
where $A$ is a real-valued $\mu$-measurable matrix function, assumed to be $Q$-periodic, symmetric, bounded and uniformly positive definite.  
The right-hand sides $f^\varepsilon$ are assumed to be divergence-free,  in the sense that 
\begin{equation}
\int_{{\mathbb R}^3}f^\varepsilon\cdot\nabla\phi\,d\mu^\varepsilon=0\qquad \forall\phi\in C_0^\infty({\mathbb R}^3),\quad \forall\varepsilon>0.
\label{div_def}
\end{equation}
For example, the case when $f^\varepsilon=f$ for all $\varepsilon>0,$ where $f$ is a continuous function with compact support, is included in the above setup.

Equation (\ref{curleq_u}) is the resolvent form of the Maxwell system in the absence of external currents, see Appendix of the present paper in addition to the above-cited monographs by Jackson and Cessenat.
In the equation (\ref{div_def}), the unknown function $u^\varepsilon$ represents the divergence-free magnetic field $H^\varepsilon$, the matrix $A$ stands the inverse of the relative dielectric permittivity 
of the medium, and 
the relative magnetic permeability 
is set to unity (so the medium is ``non-magnetic"), see Appendix for details.  The right-hand sides $f^\varepsilon$ in (\ref{curleq_u}) play an auxiliary r\^{o}le in relation to the actual electromagnetic setup: they do not appear in the original  Maxwell system but are introduced in this article for purposes of the resolvent analysis of the ``reduced'' Maxwell operator on the left-hand side of (\ref{curleq_u}). 

 Our goal is to prove order $O(\varepsilon)$ operator-norm estimates for the difference between $u^\varepsilon$ and the solution $u^\varepsilon_{\rm hom}$ of an appropriate ``homogenised equation", which is derived as part of the proof and has the form 
\begin{equation}
\label{curlhomo_u}
\curl\bigl(A^{\rm hom}\curl u^\varepsilon_{\rm hom}\bigr)+{\mathcal M}^{\rm hom}_\varepsilon u^\varepsilon_{\rm hom}=f^\varepsilon. 
\end{equation}
Here,  $A^{\rm hom}$ is a constant matrix representing the effective (``homogenised'') properties of the medium, for each $\varepsilon$ the vector-function $f^\varepsilon$ is the same as in (\ref{curleq_u}) and ${\mathcal M}^{\rm hom}_\varepsilon$ is the pseudo-differential operator with symbol $M^{\rm hom}_{\varepsilon\theta},$ which is defined in our main statement (see Theorem \ref{main_thm} below), i.e.,
\[
({\mathcal M}^{\rm hom}_\varepsilon u)(x)=\frac{1}{(2\pi)^d}\int_{{\mathbb R}^3}\int_{{\mathbb R}^3}\exp(\theta\cdot(x-y))M^{\rm hom}_{\varepsilon\theta}u(y)d\mu^\varepsilon(y)
d\theta,\quad x\in{\mathbb R}^3,\quad\qquad u\in L^2({\mathbb R}^3, d\mu^\varepsilon).
\]
In other words, we aim at finding a matrix $A^{\rm hom}$ for which there exists $C>0,$ independent of $\varepsilon$ and $f^\varepsilon,$ such that\footnote{Note that the $\varepsilon$-dependence of the homogenised solution $u^\varepsilon_{\rm hom}$ is due to the $\varepsilon$-dependence of the right-hand side $f^\varepsilon$ and the operator ${\mathcal M}^{\rm hom}_\varepsilon.$ In the case when $\mu$ is the Lebesgue measure, 
	${\mathcal M}^{\rm hom}_\varepsilon$ is the identity operator, and the $\varepsilon$-dependence of $u^\varepsilon_{\rm hom}$ is entirely due to the $\varepsilon$-dependence of the right-hand side $f^\varepsilon.$} 
\begin{equation}
\label{op_norm_estimate_ann}
\bigl\| u^\varepsilon-u^\varepsilon_{\rm hom}\bigr\|_{L^2(\R^3, d\mu^\varepsilon)}\leq C\varepsilon\|f^\varepsilon\|_{L^2(\R^3, d\mu^\varepsilon)}\qquad \forall\varepsilon\in(0,1].
\end{equation}
Clearly, a matrix $A^{\rm hom}$ with this property is unique. A similar result is obtained in \cite[Chapter 7.3]{BS04} for the case when $\mu$ is the Lebesgue measure, using perturbation analysis of the operators in (\ref{curleq_u}) near the bottom of the spectrum of the operator associated with (\ref{curleq_u}). Our approach here is based on asymptotic expansions for solutions to weak formulations, rather than the analysis of spectral properties.


Solutions of \eqref{curleq_u} are understood as pairs $(u^\varepsilon,\curl u^\varepsilon)$ in the space $H^1_{\curl}(\R^3, d\mu^\varepsilon)$ defined as the closure of the set 
\begin{equation*}
\bigl\{(\phi, \curl\phi),\;\phi\in \bigl[C^\infty_0(\R^3)\bigr]^3\bigr\}
\end{equation*}
in the direct sum $L^2(\R^3, d\mu^\varepsilon) \oplus L^2(\R^3,d\mu^\varepsilon)$. 
We say that $(u^\varepsilon, \curl u^\varepsilon)$ is a solution to \eqref{curleq_u} if
\begin{equation}\label{curleq_uweak}
\int_{\R^3}A(\cdot /\varepsilon)\curl u^\varepsilon \cdot {\curl\varphi}\,d\mu^\varepsilon +\int_{\R^3} u^\varepsilon\cdot {\varphi}\,d\mu^\varepsilon=\int_{\R^3} f^\varepsilon\cdot {\varphi}\,d\mu^\varepsilon \quad \quad \forall \varphi\in \bigl[C^\infty_0(\R^3)\bigr]^3.
\end{equation}

Clearly, the set of test functions in the identity (\ref{curleq_uweak}) can be equivalently replaced by the space $H^1_{\curl}(\R^3, d\mu^\varepsilon).$ 
Then that for each $\varepsilon>0$ the left-hand side of \eqref{curleq_uweak} defines an equivalent inner product on $H^1_{\curl}(\R^3, d\mu^\varepsilon),$ while its right-hand side can be treated as a linear bounded functional on the same.
The existence and uniqueness of $u^\varepsilon$ satisfying the integral identity (\ref{curleq_uweak}) is then a consequence of the classical Riesz representation theorem for linear functionals in a Hilbert space.

In what follows we study the resolvent of the operator $\mathcal{A}^\varepsilon$ with domain
\begin{equation}
\label{star_eq}
\begin{aligned}
{\rm dom}({\mathcal A}^\varepsilon)&=\biggl\{u\in L^2(\R^3, d\mu^\varepsilon):\ \exists\, \curl u\in L^2(\R^3, d\mu^\varepsilon)\ {\rm such\ that}\\[0.3em]
&\int_{{\mathbb R}^3}A(\cdot/\varepsilon)\curl u \cdot  {\curl\varphi}\,d\mu^\varepsilon+\int_{{\mathbb R}^3}u\cdot {\varphi}\,d\mu^\varepsilon=\int_{{\mathbb R}^3}f\cdot {\varphi}\,d\mu^\varepsilon\quad \forall\varphi\in\bigl[C^\infty_0(\R^3)\bigr]^3
\\[0.3em]
&\qquad\qquad\qquad\qquad\qquad\qquad\qquad{\rm for\; some}\ f\in L^2({\mathbb R}^3, d\mu^\varepsilon),\ {\rm div}\,\!f=0\biggr\},
\end{aligned}
\end{equation}
defined by the formula $\mathcal{A}^\varepsilon u= f-u,$ where $f\in L^2({\mathbb R}^3, d\mu^\varepsilon),$ ${\rm div}\ \!\!f=0,$ and $u\in \rm{dom}(\mathcal{A}^\varepsilon)$ are linked\footnote{It is not difficult to show that for each $u\in {\rm dom}({\mathcal A}^\varepsilon)$ there exists only one $f$ with the property described in (\ref{star_eq}).}  as in (\ref{star_eq}). Notice that, in general, for a given $u\in L^2(\R^3, d\mu^\varepsilon)$ there may be more than one element $(u, \curl u)\in H^1_{\curl}(\R^3, d\mu^\varepsilon)$. However, for each $u\in \rm{dom}(\mathcal{A}^\varepsilon)$ there exists exactly one $\curl u$ such that \eqref{star_eq} holds, which is a consequence of the uniqueness of solution to the integral identity \eqref{curleq_uweak}.

Clearly, the operator $\mathcal{A}^\varepsilon$ is symmetric. Furthermore, similarly to \cite{CD18} we infer that $\rm{dom}(\mathcal{A}^\varepsilon)$ is dense in 
\[
\bigl\{u\in L^2(\R^3, d\mu^\varepsilon): \divv u=0\bigr\}.
\]
Indeed, by the definition of $\rm{dom}(\mathcal{A}^\varepsilon),$ see (\ref{star_eq}), if $f \in L^2(\R^3,d\mu^\varepsilon),$ ${\rm div}\,\!f=0$ and $u,v\in \rm{dom}(\mathcal{A}^\varepsilon)$ are such that $\mathcal{A}^\varepsilon u+u=f$ and $\mathcal{A}^\varepsilon v+v=u,$ one has
\begin{equation*}
\int_{\R^3} |u|^2 d\mu^\varepsilon=\int_{\R^3} f\cdot {v} \;d\mu^\varepsilon.
\end{equation*}
This identity entails that if $f$ is orthogonal to $\rm{dom}(\mathcal{A}^\varepsilon)$, then $u=0$ and so $f=0$. It follows from the definition of ${\mathcal A}^\varepsilon$ that its defect numbers are zero, hence it is self-adjoint.
Analogously, we define the operator $\mathcal{A}^{\rm hom}$ associated with the problem \eqref{curlhomo_u}, so that \eqref{curlhomo_u} holds if and only if $u^\varepsilon_{\rm hom}=(\mathcal{A}^{\rm hom}+I)^{-1} f^\varepsilon$.


 All integrals and differential operators below, unless indicated otherwise, are understood appropriately with respect to the measure $\mu$.
Throughout the paper we use the notation $e_\kappa$ for the exponent $\exp({\rm i}\kappa\cdot y)$, $y\in Q,$ $\kappa\in[-\pi, \pi)^3,$ and a similar notation  $e_\theta$ for the exponent $\exp({\rm i}\theta\cdot x),$ $x\in {\mathbb R}^3,$ $\theta\in\varepsilon^{-1}[-\pi, \pi)^3$.
We denote by $C_\#^\infty$ the set of $Q$-periodic functions in $C^\infty({\mathbb R}^3)$, and $\curl \phi$, $\curl(e_\kappa\phi)$ $\curl (e_{\varepsilon\theta}\phi)$ are the classical curls of smooth vector functions $\phi,$ $e_\kappa\phi,$ $e_{\varepsilon\theta}\phi.$ Finally, we denote by $L^2(Q,d\mu)$ the space of $\C^3$-valued functions 
that are square integrable over $Q$ with respect to the measure $\mu.$ 

The structure of the paper is as follows. In order to formulate the system of Maxwell equations in the setting of singular periodic structures, we introduce the notion of weak differentiability for functions that are square integrable with respect to a general Borel measure. In our approach to this task we follow the works \cite{Zhikov2000}, \cite{Zhikov2002} and \cite{ZhikovNote}. In Section \ref{Sobolev_section} we define the Sobolev spaces with respect to an arbitrary Borel measure and highlight some of their properties.
In Section \ref{Floquet_section} we introduce a suitable version of the classical Floquet transform and write a direct integral representation for the resolvents of 
the operators ${\mathcal A}^\varepsilon$ in terms of the resolvents of operators in $L^2(Q, d\mu),$ which are equivalently represented by the problems (\ref{ft_problemweak}) depending on the fibre parameter $\theta$ (``quasimomentum"). In Section \ref{Helmholtz_section} we extend the classical Helmholtz decomposition to the case of functions in $L^2(Q, d\mu)$ and introduce an appropriate generalisation of the Poincar\'{e} inequality, which we subsequently demonstrate to be sufficient for the norm-resolvent asymptotic analysis of the problems (\ref{ft_problemweak}). Section \ref{section_as} and Section \ref{proof_sec} cover the proof of our main results, Theorem \ref{main_thm} and Theorem \ref{corollary_main_thm}. This involves the analysis of a suitable asymptotic representation for the (parameter-dependent) solution to (\ref{ft_problemweak}) and a proof, based on our new quantitative tools, of remainder estimates for the difference between the solution and the leading-order term of the asymptotics. As in terms of the original ``physical" Maxwell system our main result is formulated for the magnetic component of the electromagnetic field, in Section \ref{electric_sec} we discuss how this translates to similar statements for the  electric field and displacement. Finally, in Appendix we discuss in more detail how the equation (\ref{curleq_u}) emerges from the dimensional analysis of the equations of electromagnetism. A reader wishing to get a better idea of the physical underpinnings of our analysis, may wish to inspect this appendix first.

\section{Sobolev spaces of quasiperiodic functions}
\label{Sobolev_section}




The aim of this section is to describe the functional analytic framework for our study of the problem (\ref{curleq_u}).
As a particular case of the notion of ``weak differentiability" of square-integrable vector functions with respect an arbitrary Borel measure, 
we introduce a suitable generalisation of the 
classical curl operator. In what follows, $\mu$ is an arbitrary $Q$-periodic Borel measure.

 
\begin{defi}\label{defiH1curl}
The space $H^1_{\curl}$
is defined as the closure of the set 
\begin{equation}
\bigl\{(\phi,\curl\phi),\ \phi\in[C^\infty_\#]^3 
\bigr\}
\label{Hcurldef}
\end{equation}
 in the product $L^2(Q, d\mu;\C^3)\times L^2(Q, d\mu;\C^3).$
\end{defi}

Elements of the closure (\ref{Hcurldef}) are pairs $(u,v)$, where $u, v\in L^2(Q, d\mu),$ such that
\begin{equation}\label{propH1curl}
\exists\,\{\phi_n\}\subset\bigl[C^\infty_\#\bigr]^3: \quad \quad \int_Q |\phi_n-u|^2 d\mu\stackrel{n\to\infty}{\longrightarrow}0 \quad \quad \int_Q |\curl \phi_n-v|^2 d\mu\stackrel{n\to\infty}{\longrightarrow}0
\end{equation}
The element\footnote{For a general measure $\mu,$ a vector $u\in L^2(Q, d\mu)$ has multiple curls with respect to $\mu.$ In particular, any vector $g\in L^2(Q,d\mu)$ with the property
\begin{equation*}
\exists\,\{\phi_n\}\subset[C^\infty_\#(Q)]^3: \quad
\int_Q |\phi_n|^2d\mu\stackrel{n\to\infty}{\longrightarrow}0, \quad \quad \int_Q |g-\curl\phi_n|^2d\mu\stackrel{n\to\infty}{\longrightarrow}0
\end{equation*}
is a curl with respect to $\mu$ of the zero vector, {\it i.e.} a ``curl of zero".} $v$ in \eqref{propH1curl} is referred to as a curl of $u$ with respect to $\mu.$ 
We will often use the notation $\curl u$ without indicating the measure $\mu$ explicitly, assuming that it is clear from the context what the measure is.



We now extend to the vector setting (see {\it e.g.} \cite{CD18} for the scalar case) the definition of the Sobolev  space of quasiperiodic functions with respect to 
the measure $\mu.$ 
\begin{defi}
\label{H1kappa_def}
For each $\kappa\in [-\pi, \pi)^3=:Q',$ the space $H^1_{\curl,\kappa}$ is defined as the closure of the set ({\it cf.} (\ref{Hcurldef})) 
\begin{equation}
\bigl\{\bigl(e_\kappa \phi, \curl(e_\kappa\phi)\bigr): \phi\in [C_\#^\infty]^3\bigr\}
\label{curl_def_again}
\end{equation}
with respect the standard norm in $L^2(Q, d\mu) \oplus L^2(Q, d\mu)$. For $(u,v) \in H^1_{\curl,\kappa},$ we denote by $\curl(e_\kappa u)$ the second element $v$ in the pair, which we sometimes refer to as a ``$\kappa$-curl of $u$." We will continue using the notation 
$H^1_{\curl}$ (see Definition \ref{defiH1curl}) for the space $H^1_{\curl,\kappa}$ with $\kappa=0.$
\end{defi}
Note that there may be different elements in $H^1_{\curl,\kappa}$ with the same first component. Indeed, for any pair $(u,v) \in H^1_{\curl,\kappa}$ and a vector function $w$ obtained as the limit in 
$L^2(Q, d\mu)$ of $\curl(e_\kappa\phi_n)$ for a sequence $\{\phi_n\}\subset [C_\#^\infty]^3$ converging to zero in $L^2(Q, d\mu)$, the element $(u, v+w)$ is also in $H^1_{\curl,\kappa}$.
In addition, $H^1_{\curl,\kappa}$ and $H^1_{\curl, 0}$ are related by a one-to-one map. Indeed, for any $(u, v)\in H^1_{\curl,\kappa},$ the pair ${({\overline{e}_\kappa}u,\,{\overline{e}_\kappa}(v-{\rm i}\kappa\times u))}$ is an element of $H^1_{\curl},$ which follows from 
\[
\curl \phi_n=\curl(\overline{e}_\kappa e_\kappa\phi_n)=\overline{e}_\kappa\curl(e_\kappa\phi_n)-{\rm i}\kappa\times\phi_n,
\]
for all sequences $\{\phi_n\}\subset[C_\#^\infty]^3$ such that $e_\kappa\phi_n\to0,$ $\curl(e_\kappa\phi_n)\to0$ in $L^2(Q, d\mu)$ as $n\to\infty.$ Conversely, for all $(\widetilde{u}, \widetilde{v})\in H_{\curl}^1$ one has $\widetilde{v}={\overline{e}_\kappa}(v-{\rm i}\kappa\times u)$ for some $(u, v)\in H^1_{\curl,\kappa}$.

We say that $F\in L^2(Q,d\mu)$ is divergence-free (more precisely, ${\rm div}_\kappa$-free), or solenoidal, and write ${\overline{e}_\kappa}{\rm div}(e_\kappa F)=0,$ if
\begin{equation}
\int_Qe_\kappa F\cdot {\nabla(e_\kappa\phi)}\,d\mu=0\qquad \forall\phi\in C_\#^\infty.
\label{weak_div_cond}
\end{equation}
Now suppose that $A$ is a $\mu$-measurable, $\mu$-essentially bounded, symmetric, pointwise positive real-valued matrix function such that $A^{-1}$
is $\mu$-essentially bounded. For each $\kappa\in Q'$ we analyse the operator $\mathcal{A}_\kappa$ with domain ({\it cf.} (\ref{star_eq}))
\begin{equation}
\label{Akappa_def}
\begin{aligned}
{\rm dom}({\mathcal A}_\kappa)&=\biggl\{u\in L^2(Q, d\mu):\ \exists\, \curl(e_\kappa u)\in L^2(Q, d\mu)\ {\rm such\ that}
\\[0.3em]
&\int_QA\curl (e_\kappa u) \cdot  {\curl(e_\kappa\varphi)}\,d\mu+\int_Q u\cdot {\varphi}\,d\mu=\int_Q F\cdot {\varphi}\,d\mu\quad \forall \varphi \in\bigl[C_\#^\infty\bigr]^3\\[0.3em]
&\qquad\qquad\qquad\qquad\qquad\qquad{\rm for\ some}\ F\in L^2(Q, d\mu),\ \ {\overline{e}_\kappa}{\rm div}(e_\kappa F)=0\biggr\},
\end{aligned}
\end{equation}
defined by the formula $\mathcal{A}_\kappa u= F-u,$ where $F\in L^2(Q, d\mu)$ and $u\in \rm{dom}(\mathcal{A}_\kappa)$ are linked as in (\ref{Akappa_def}). By an argument similar to the case of $\mathcal{A}^\varepsilon$, the domain $\rm{dom}(\mathcal{A}_\kappa)$ is dense in 
\[
\bigl\{u\in L^2(Q, d\mu): {\overline{e}_\kappa}\divv(e_\kappa u)=0\bigl\},
\] 
and $\mathcal{A}_\kappa$ is self-adjoint.


\section{Floquet transform}
\label{Floquet_section}
In this section we define, similarly to the scalar case discussed in \cite{CD18}, a representation for functions in $L^2(\R^3, d\mu^\varepsilon)$ that is unitarily equivalent to Gelfand transform \cite{Gelfand}. In the paper \cite{ZP16}, properties of the Gelfand transform with respect to the arbitrary periodic Borel measure $\mu$ have been studied and their applications to spectral analysis of elliptic PDE have been discussed. Here we describe its ``Floquet version", which is unitary equivalent 
via a multiplication by the function $e_\kappa$ (whose $L^2$ norm is clearly unity).
\begin{defi}
For $\varepsilon>0$ and $u\in [C^\infty_0(\R^3)]^3$, the $\varepsilon$-Floquet transform ${\mathcal F}_\varepsilon u$ is the function  
\[
({\mathcal F}_\varepsilon u)(y, \theta)=\biggl(\frac{\varepsilon^2}{2\pi}\biggr)^{3/2}\sum_{n\in {\mathbb Z}^3}u(\varepsilon y+\varepsilon n)\exp(-{\rm i}\varepsilon n\cdot\theta),\qquad y\in Q,\quad \theta\in\varepsilon^{-1}Q'=\varepsilon^{-1}[-\pi, \pi)^3.
\]
\end{defi}
Note for a given $u\in [C^\infty_0(\R^3)]^3,$ the function ${\mathcal F}_\varepsilon u={\mathcal F}_\varepsilon u(y, \theta)$ is $\varepsilon\theta$-quasiperiodic on $Q$ as a function of $y$ and $\varepsilon^{-1}Q'$-periodic as a function of $\theta.$  The mapping $\mathcal{F}_\varepsilon$ preserves the norm and can be extended to an isometry 
\[
L^2(\R^3, d\mu^\varepsilon)\longrightarrow L^2(Q\times\varepsilon^{-1}Q', d\mu\times d\theta),
\]
for which we keep the same notation $\mathcal{F}_\varepsilon$ and the term ``$\varepsilon$-Floquet transform". By an argument similar to that given in \cite[Section 3]{CD18}, 
the mapping $\mathcal{F}_\varepsilon$ is shown to be unitary\footnote{Note that in \cite[Section 3]{CD18}, the scalar version of the transform we denote here by $\mathcal{F}_\varepsilon$ was introduced as a product of two unitary transforms, one of which was labelled by 
$\mathcal{F}_\varepsilon$ while the other was a unitary rescaling.} for all $\varepsilon>0,$ and its inverse
is given by the formula
\begin{equation*}
({\mathcal F}_\varepsilon^{-1}g)(x)=(2\pi)^{-3/2}\int_{\varepsilon^{-1}Q'}g\biggl(\frac{x}{\varepsilon}, \theta\biggr)\,d\theta,\quad x\in {\mathbb R}^3\quad \qquad \forall g\in L^2(Q\times \varepsilon^{-1}Q', d\mu\times d\theta),
\end{equation*}
where for each $\theta\in\varepsilon^{-1}Q'$ the function $g\in L^2(Q\times\varepsilon^{-1}Q', d\mu\times d\theta)$ is extended as a $\theta$-quasiperiodic function to the whole of ${\mathbb R}^3$ so that 
\[
g(z, \theta)=\widetilde{g}(z, \theta)\exp({\rm i}z\cdot\theta),\quad z\in{\mathbb R}^3,\qquad \widetilde{g}(\cdot, \theta)\ \ Q{\text{\rm -periodic.}}
\]



As a result of applying the transform ${\mathcal F}_\varepsilon$ to the operator ${\mathcal A}_\varepsilon$ of the problem (\ref{curleq_u}), we obtain the following representation for the resolvent of ${\mathcal A}_\varepsilon.$
\begin{prop}
\label{propequivalence}
For each $\varepsilon>0,$ the following unitary equivalence between the resolvent of the operator $\mathcal{A}^\varepsilon$ and the direct integral of the family of resolvents for $\mathcal{A}_{\varepsilon\theta}$, $\theta\in \varepsilon^{-1} Q',$ holds:
\[
({\mathcal A}^\varepsilon+I)^{-1}={\mathcal F}_\varepsilon^{-1}
\biggl(\int_{\varepsilon^{-1}Q'}^\oplus e_{\varepsilon\theta}(\varepsilon^{-2}{\mathcal A}_{\varepsilon\theta}+I)^{-1}\overline{e}_{\varepsilon\theta}\,d\theta\biggr)
{\mathcal F}_\varepsilon,
\]
where $\overline{e}_{\varepsilon\theta},$ $e_{\varepsilon\theta}$  represent the operators of multiplication by  
$\overline{e}_{\varepsilon\theta},$ $e_{\varepsilon\theta},$ respectively.
\end{prop}
\begin{proof}[Sketch of the proof]
The argument is similar to that given in \cite{CC16}, \cite{CD18} for the scalar case. We consider the solution $(u^\varepsilon, \curl u^\varepsilon)\in H^1_{\curl}$ of the problem \eqref{curleq_u} with $f^\varepsilon\in[C^\infty_0(\R^3)]^3.$ 
We then introduce the ``periodic amplitude" of its $\varepsilon$-Floquet transform
\begin{equation}\label{u_theta^varepsilon}
u_\theta^\varepsilon(y):= \overline{e}_{\varepsilon\theta}
{\mathcal F}_\varepsilon u^\varepsilon(y)=\biggl(\frac{\varepsilon^2}{2\pi}\biggr)^{3/2}\sum_{n\in{\mathbb Z}^3}u^\varepsilon(\varepsilon y+\varepsilon n)\exp\bigl(-{\rm i}(\varepsilon y+\varepsilon n)\cdot\theta\bigr),\quad y\in Q.
\end{equation}
By approximating $u_\theta^\varepsilon$ with smooth functions, it is straightforward to see that if, for each 
choice of $\curl u^\varepsilon,$ we write 
\[
\curl(e_{\varepsilon\theta}u_\theta^\varepsilon)(y)=\varepsilon\biggl(\frac{\varepsilon^2}{2\pi}\biggr)^{3/2}\sum_{n\in{\mathbb Z}^3}\curl u^\varepsilon(\varepsilon y+\varepsilon n)\exp\bigl(-{\rm i}\varepsilon n\cdot\theta\bigr),\qquad y\in Q,
\]
then 
$(e_\kappa u^\varepsilon_\theta, \curl(e_\kappa u^\varepsilon_\theta)) \in H^1_{\curl,\kappa} (Q, d\mu)$. Furthermore,
\begin{equation}
\begin{aligned}
\varepsilon^{-2}\int_QA\curl(e_{\varepsilon\theta}u_\theta^\varepsilon)\cdot {\curl(e_{\varepsilon\theta}\varphi)}\,d\mu&+\int_Qe_{\varepsilon\theta}u_\theta^\varepsilon\cdot {e_{\varepsilon\theta}\varphi}\,d\mu\\[0.5em]
&=\int_Qe_{\varepsilon\theta}F\cdot {e_{\varepsilon\theta}\varphi}\,d\mu\qquad \forall\varphi\in\bigl[C_\#^\infty\bigr]^3,
\end{aligned}
\label{ft_problemweak}
\end{equation}
where $F:= \overline{e}_{\varepsilon\theta}
\mathcal{F}_\varepsilon f$. It is verified directly that $F$ is solenoidal, {\it cf.} (\ref{weak_div_cond}).
By the density of $f\in[C^\infty_0(\R^3)]^3$ in $L^2(\R^3,d\mu^\varepsilon)$ (see {\it e.g.} \cite[Chapter 9]{Makarov_Podkorytov}),
we obtain the claim. 
\end{proof}

In what follows, we study the asymptotic behaviour, as $\varepsilon\to0,$ of the solutions $u_\theta^\varepsilon$ to the problems
\begin{equation}
\label{ft_problem}
\varepsilon^{-2} \overline{e}_{\varepsilon\theta} \curl\bigl(A \curl( e_{\varepsilon\theta} u_\theta^\varepsilon)\bigr) + u_\theta^\varepsilon=F \quad\quad \varepsilon >0, \quad \theta\in\varepsilon^{-1} Q',
\end{equation}
for all solenoidal $F\in L^2(Q, d\mu).$ 
For each right-hand side $F,$ problem (\ref{ft_problem}) is understood in the sense of the identity (\ref{ft_problemweak}).
We will show that $u_\theta^\varepsilon$ is $\varepsilon$-close with respect to the norm of $L^2(Q, d\mu)$, uniformly in 
$\theta\in\varepsilon^{-1}Q'$ to the constant vector $c^\varepsilon_\theta$ solving the ``homogenised" equation associated with \eqref{ft_problem}:
\begin{equation}
\label{homo_problem_ft}
\theta\times  A^{\rm hom}(\theta \times c^\varepsilon_\theta) +M^{\rm hom}_{\varepsilon\theta}c^\varepsilon_\theta=\int_Q F d\mu, \quad \quad \theta\in \varepsilon^{-1} Q',
\end{equation}
where the matrices $A^{\rm hom},$ $M^{\rm hom}_\kappa,$ $\kappa\in[-\pi,\pi)^3,$ will be defined in Section \ref{section_as} (see, in particular, the formula (\ref{Ahom_formula}).)
Note that by setting $\phi=1$ in (\ref{weak_div_cond}) one infers that
\[
\theta\cdot\int_Q F \;d\mu=0.
\] 
We will use this observation in the proof (Section \ref{proof_sec}) of the estimate stated in the main result,  Theorem \ref{corollary_main_thm}.


\section{Quasiperiodic 
Helmholtz decomposition }
\label{Helmholtz_section}
In the asymptotic analysis of systems of Maxwell equations, the Helmholtz (or Weyl, or Hodge) decomposition \cite[Chapter 2]{Cessenat}, \cite[Chapter 9]{DL}, \cite[Section 3.7]{Monk} for square-integrable functions proves useful. It provides a convenient geometric interpretation of the degeneracy in the problem, namely the fact that the differential expression vanishes on the infinite-dimensional space of gradients of $H^2$ functions, which suggests representing the relevant $L^2$ space as an orthogonal sum of curl-free functions with zero mean, divergence-free functions with zero mean and constants. In the present work we require a special version of the decomposition, which takes into account the quasiperiodicity of the functions involved and also incorporates a class of periodic Borel measures for the underlying $L^2$ space.  

Before formulating the next proposition, we recall that, similarly to the construction of Section \ref{Sobolev_section}, the notions of a gradient of a quasiperiodic $L^2$ function with respect to the measure $\mu$ and the associated Sobolev spaces $H^1_{\kappa}$ of $\kappa$-quasiperiodic functions, $\kappa\in Q',$ as well as the Sobolev space of periodic functions $H^1_\#,$ with respect to the measure $\mu$ can be defined. In particular, for each $\kappa\in [-\pi, \pi)^3=:Q',$ the space $H^1_{\kappa}$ is defined as the closure of the set ({\it cf.} (\ref{curl_def_again})) 
\[
\bigl\{\bigl(e_\kappa \phi, \nabla(e_\kappa\phi)\bigr): \phi\in C_\#^\infty \bigr\}
\] 
with respect the standard norm in $L^2_{\rm s}(Q, d\mu) \oplus L^2(Q, d\mu)$,
 where $L^2_{\rm s}(Q, d\mu)$ is the space of ${\mathbb C}$-valued functions on $Q$ that are square integrable with respect to the measure $\mu$ (so that $L^2(Q, d\mu)=[L^2_{\rm s}(Q, d\mu)]^3.$) For $(u,v) \in H^1_{\kappa},$ we denote by 
 $\nabla(e_\kappa u)$ the second element $v$ in the pair and use the notation 
$H^1_\#$ for the space $H^1_{\kappa}$ with $\kappa=0.$  We do not dwell on these definitions further and instead refer the reader to \cite{CD18}.  



Denote by $C^\infty_{\#,0}$ the set of infinitely smooth $Q$-periodic functions with zero mean over $Q$, and by $H^1_{\#,0}$ 
the subspace of $H^1_\#$ consisting of functions with zero mean over $Q$. A key ingredient of our generalisation of the Helmholtz decomposition is the following construction.

\begin{prop}\label{prop_eqPhi}
Suppose that $u\in L^2(Q, d\mu).$ The problem
\begin{align}
{\overline{e}_\kappa}\triangle (e_\kappa \Phi_u) = {\overline{e}_\kappa}\divv(e_\kappa u), \label{eq_Phi}
\end{align}
understood in the sense that 
\begin{equation}\label{eq_Phi_weak}
\int_Q \nabla(e_\kappa \Phi_u)\cdot  {\nabla(e_\kappa \varphi)}\; d\mu = \int_Q e_\kappa u\cdot {\nabla(e_\kappa \varphi)} \; d\mu \quad \quad \forall \varphi\in C^\infty_{\#,0},
\end{equation}
has a unique solution $\Phi_u \in H^1_{\#,0}.$
\end{prop}
\begin{proof}
Considering the sesquilinear form on the left-hand side of \eqref{eq_Phi_weak}, the existence and uniqueness of solution $\Phi_u$ follows from the Lax-Millgram theorem, see {\it e.g.} \cite{BLP}. Indeed, the continuity of the form is obtained by setting $\nabla (e_\kappa u)=e_\kappa ({\rm i}\kappa u+\nabla u)$ for all scalar functions $u\in H^1_\#.$ The coercivity is a consequence of the $\kappa$-uniform Poincar\'{e} inequality proved in \cite{CD18} in the scalar setting. 
\end{proof}

Using the above statement, for each $u\in L^2(Q,d\mu)$ we write
\begin{equation}
\label{decomposition}
u=\widetilde{u}+\int_Q u+{\overline{e}_\kappa}\nabla (e_\kappa \Phi_u),
\end{equation}
where the function $\widetilde{u}$ satisfies the following conditions on its divergence and mean:
\begin{align}
&{\overline{e}_\kappa}\divv \biggl(e_\kappa\biggl(\widetilde{u}+\int_Q u\biggr)\biggr)=0, \label{eq_utilde}\\
&\int_Q\bigl(\widetilde{u}+{\overline{e}_\kappa}\nabla (e_\kappa \Phi_u)\bigr)=0. \label{mean_utilde_Phi}
\end{align}
The uniqueness part of Proposition \ref{prop_eqPhi} implies that there is a unique function $\Phi_u$ with zero mean such that \eqref{decomposition} holds, and hence $\widetilde{u}$ is also defined uniquely.

In what follows we make the following assumption about the measure $\mu.$ 
\begin{assumption}
\label{ass1}
There exists $C_{\rm P}>0$  such that for all 
$\kappa\in Q'$ and $(e_\kappa u, \curl(e_\kappa u))\in H^1_{\curl,\kappa}$
the following Poincar\'{e}-type inequality holds: 
\begin{equation}
\label{curlpoincare}
\biggl\|u-\int_Q u-{\overline{e}_\kappa}\nabla(e_\kappa \Phi_u)+\int_Q{\overline{e}_\kappa}\nabla(e_\kappa \Phi_u)\biggr\|_{L^2(Q,d\mu)}\leq C_{\rm P}\bigl\|\curl(e_\kappa u)\bigr\|_{L^2(Q,d\mu)}.
\end{equation}
\end{assumption}


\begin{rmk}
For each fixed $(e_\kappa u, \curl(e_\kappa u))\in H^1_{\curl,\kappa},$ denote 
\[
{\mathfrak u}:=u-{\overline{e}_\kappa}\nabla(e_\kappa \Phi_u),
\]
and notice that $\curl(e_\kappa u)$ is one of the $\kappa$-curls of the function ${\mathfrak u}$ thus defined, since zero is one of the $\kappa$-curls of ${\overline{e}_\kappa}\nabla(e_\kappa \Phi_u).$
Then one has $\overline{e}_\kappa{\rm div}\,(e_\kappa{\mathfrak u})=0,$ and (\ref{curlpoincare}) takes the form 
\begin{equation}
\biggl\|{\mathfrak u}-\int_Q{\mathfrak u}\biggr\|_{L^2(Q,d\mu)}\leq C_{\rm P}\bigl\|\curl(e_\kappa {\mathfrak u})\bigr\|_{L^2(Q,d\mu)}
\label{poincare_v}
\end{equation}
\end{rmk}
It can be shown that the following periodic measures satisfy Assumption \ref{ass1} (and, equivalently, the Poincar\'{e} inequality (\ref{poincare_v})): 

(a) Consider a finite set $\{{\mathcal P}_j\}_{j=1}^N$ of planes each of which is orthogonal to one of the coordinate axes and such that 
$(\cup_{j=1}^N{\mathcal P}_j)\cap Q$ is non-empty and connected. Define the measure $\mu$ on $Q$ by the formula 
\[
\mu(B)=N^{-1}\sum_j\vert {\mathcal P}_j\cap B\vert_2\ \ {\rm for\ all\ Borel\ } B\subset Q,
\]  
where $\vert\cdot\vert_2$ represents the $2$-dimensional Lebesgue measure, {\it i.e.} $\vert{\mathcal P}_j\cap B\vert_2$ is the area of ${\mathcal P}_j\cap B.$ In other words, $\mu$ is the two-dimensional Hausdorff measure on $(\cup_{j=1}^N{\mathcal P}_j)\cap Q,$ normalised by $N=\sum_j\vert {\mathcal P}_j\cap Q\vert_2.$ 

(b) The suitably normalised two-dimensional Hausdorff measure on the intersection with $Q$ of a rigid rotation in ${\mathbb R}^3$ of the union $\cup_{j=1}^N{\mathcal P}_j$ described in a.

(c) The suitably normalised two-dimensional Hausdorff measure on a finite union of sets from the class described in b, under the condition that the union is connected.

(d) The (three-dimensional) Lebesgue measure on $Q.$


(e) Consider a finite set  $\{\mu_j\}_{j=1}^M$ of measures satisfying any of the conditions a, b, d, such that the union of the supports $S_j:={\rm supp}(\mu_j),$ $j=1,\dots, M,$ is connected. Define the measure $\mu$ by the formula
\[
\mu(B)=\frac{\sum_{j=1}^M\mu_j(S_j\cap B)}{\sum_{j=1}^M\mu_j(S_j)}\ \ {\rm for\ all\ Borel\ } B\subset Q.
\]  
(Note that c is a particular case of e.)





\section{Asymptotic approximation for $u_\theta^\varepsilon$}
\label{section_as}

In what follows, we often drop the second component vectors in $H^1_{\curl}$ or $H^1_{\curl, \kappa}$ and write $u\in H^1_{\curl},$ and $u\in H^1_{\curl, \kappa}$ meaning ``the first component $u$ of an element of $H^1_{\curl}$" and  
``the first component $u$ of an element of $H^1_{\curl, \kappa}$", respectively.

 In order to write an asymptotic expansion for the solution $u^\varepsilon_\theta$ of \eqref{ft_problem}, we consider the following ``cell problem" ({\it cf.} \cite{CE16}) for a matrix-valued function $\widehat{N}:$
\begin{equation}
\label{curl_cellpb}
\curl(A\curl\widehat{N})=-\curl A,\qquad \divv\widehat{N}=0,\qquad \int_Q\widehat{N}=0,
\end{equation}
where $(\curl A)_{ij}=\epsilon_{ilk}A_{kj, l}$ (and similarly $(\curl \widehat{N})_{ij}=\epsilon_{ilk}\widehat{N}_{kj,l})$ and $(\divv \widehat{N})_i=\widehat{N}_{si,s},$ $i, j=1,2,3,$ where $(\epsilon_{ijk})_{i, j,k=1}^3$ is the Levi-Civita tensor.

The first equation in (\ref{curl_cellpb})  is understood in the sense of the integral identity
\begin{equation}
\label{weak_curlcellpb}
\int_Q A \curl \widehat{N}\,{\curl \varphi}= -\int_Q A\, {\curl \varphi} \quad \quad \forall \varphi \in \bigl[C_\#^\infty\bigr]^3.
\end{equation}
\begin{prop}
There exists a unique matrix-valued function $\widehat{N}$ with columns in $H^1_{\curl}$ that solves \eqref{curl_cellpb}. 
\end{prop}
\begin{proof} 
It follows from Assumption \ref{ass1} with $\kappa=0$
that the skew-symmetric sesquilinear form 
\[
\int_Q A \curl u\cdot {\curl v},\qquad u, v\in H^1_{\curl}\cap\biggl\{u: \divv u=0, \int_Q u=0\biggr\},
\]
is coercive. Noting also that it is also clearly continuous,
the claim follows by the Riesz representation theorem.
\end{proof}

In terms of the family of $\theta$-parametrised problems (\ref{ft_problem}) out main result is stated as follows.
\begin{thm}
\label{main_thm}
For each $\varkappa\in Q',$ denote by $\Psi_\kappa$ the vector with components in $H^1_{\#,0}$ that solves
\begin{equation}
	\overline{e}_\kappa{\rm div}\bigl(\nabla(e_\kappa\Psi_\kappa)+e_\kappa I\bigr)=0.
	\label{Psi_equation}
\end{equation}
The following estimate holds for the solutions to \eqref{ft_problem} with a constant $C>0$ independent of $\varepsilon,\; \theta, \; F$:
\begin{equation}\label{main_stat}
\Bigl\|u_\theta^\varepsilon-\bigl(\overline{e}_{\varepsilon\theta}\nabla (e_{\varepsilon\theta}\Psi_{\varepsilon\theta})+I\bigr)c^\varepsilon_\theta\Bigr\|_{L^2(Q,d\mu)}\leq C \varepsilon \|F\|_{L^2(Q,d\mu)},
\end{equation}
where $c^\varepsilon_\theta$ is the vector solution of the homogenised problem \eqref{homo_problem_ft}, that is
\begin{equation}\label{c_theta}
c^\varepsilon_\theta=c^\varepsilon_\theta(F)=\bigl({\mathfrak A}^{\hom}_\theta+M^{\rm hom}_{\varepsilon\theta}\bigr)^{-1}\int_Q F.
\end{equation}
Here ${\mathfrak A}^{\hom}_\theta$
is the matrix-valued quadratic form given by the first term on the left-hand side of equation \eqref{homo_problem_ft}, where 
\begin{equation}
   A^{\rm hom}:= \int_Q A(\curl\widehat{N}+I),
\label{Ahom_formula}
\end{equation}
and for each $\kappa\in[-\pi,\pi)^3,$ the matrix $M^{\rm hom}_\kappa$ is given by 
\[
M^{\rm hom}_\kappa:=\int_Q\bigl(\overline{e}_{\kappa}\nabla(e_{\kappa}\Psi_{\kappa})+I\bigr)d\mu.
\]
\end{thm}
\begin{rmk}
For each $\theta\in{\mathbb R}^3$ the matrix ${\mathfrak A}^{\hom}_\theta$ can be written as
\begin{equation*}
({\mathfrak A}^{\hom}_\theta)_{ij}=A^{\rm hom}_{kl}\epsilon_{kit}\epsilon_{lsj}\theta_t\theta_s,\qquad i,j=1,2,3,
\end{equation*}
where, as above, $(\epsilon_{ijk})_{i, j,k=1}^3$ is the Levi-Civita tensor.
\end{rmk}

Combined with Proposition \ref{propequivalence}, the uniform estimate (\ref{main_stat}) yields the following result, announced in Introduction, {\it cf.}
 (\ref{op_norm_estimate_ann}).
\begin{thm}
\label{corollary_main_thm}
There exists $C>0,$ independent of $\varepsilon$ and the choice of $f^\varepsilon\in L^2(\R^3, d\mu^\varepsilon),$ such that
\begin{equation}
\|u^\varepsilon-u^\varepsilon_{\rm hom}\|_{L^2(\R^3, d\mu^\varepsilon)}\leq C\varepsilon \|f^\varepsilon\|_{L^2(\R^3, d\mu^\varepsilon)},
\label{corollary_estimate}
\end{equation}
where $u^\varepsilon$ are the solutions of the original problem \eqref{curleq_u}, and $u^\varepsilon_{\rm hom}$ is the solution of the homogenised equation \eqref{curlhomo_u}.
\end{thm}
\begin{proof}
Throughout the proof we shall drop the superscript $\varepsilon$ in $f^\varepsilon$ for brevity. For each element of the sequence $f=f^\varepsilon\in L^2(\R^3, d\mu^\varepsilon),$ consider the $Q$-periodic function $f^\varepsilon_\theta:= \overline{e}_{\varepsilon\theta}
\mathcal{F}_\varepsilon f,$ {\it cf.} (\ref{u_theta^varepsilon}), so that
\begin{equation*}
\int_Q f^\varepsilon_\theta d\mu=\widehat{f}(\theta,\varepsilon),  \quad \theta\in \varepsilon^{-1}Q',\quad \hbox{where} \quad \widehat{f}(\theta,\varepsilon):= (2\pi)^{-3/2}\int_{\R^3}\overline{e}_\theta fd\mu^\varepsilon, \quad \theta\in \R^3.
\end{equation*}
For each $\theta\in\varepsilon^{-1}Q',$ consider the solution $u^\varepsilon_\theta$ to the problem \eqref{ft_problem} with $F=f^\varepsilon_\theta$.
Using Proposition \ref{propequivalence}, we can write the difference between the solutions $u^\varepsilon$ and $u^\varepsilon_{\rm hom}$ to \eqref{curleq_u} and \eqref{curlhomo_u}, respectively, as
\begin{align}
u^\varepsilon-u^\varepsilon_{\rm hom} &= (\mathcal{A}^\varepsilon+I)^{-1} f- (\mathcal{A}^{\rm hom} +I)^{-1} f\nonumber
\\[0.5em]
&= \mathcal{F}_\varepsilon^{-1}
e_{\varepsilon\theta} (\varepsilon^{-2} \mathcal{A}_{\varepsilon\theta}+I)^{-1} f^\varepsilon_\theta -(\mathcal{A}^{\rm hom} +I)^{-1} f
=\mathcal{F}_\varepsilon^{-1}
e_{\varepsilon\theta} u^\varepsilon_\theta -(\mathcal{A}^{\rm hom} +I)^{-1} f\nonumber
\\[0.5em]
&=\bigl(\mathcal{F}_\varepsilon^{-1}
e_{\varepsilon\theta} u^\varepsilon_\theta - \mathcal{F}_\varepsilon^{-1}
e_{\varepsilon\theta} c_\theta^\varepsilon\bigr) + \bigl(\mathcal{F}_\varepsilon^{-1}
e_{\varepsilon\theta} c_\theta^\varepsilon - (\mathcal{A}^{\rm hom} +I)^{-1} f\bigr).\label{last_line_diff}
\end{align}
For the  first term $\mathcal{F}_\varepsilon^{-1}
e_{\varepsilon\theta} u^\varepsilon_\theta - \mathcal{F}_\varepsilon^{-1}
e_{\varepsilon\theta} c_\theta^\varepsilon,$ we can use the Theorem \ref{main_thm}, since $\mathcal{F}_\varepsilon$
and the multiplication by $e_{\varepsilon\theta}$ are unitary operators. The second term in (\ref{last_line_diff}) can be written as 
\begin{align*}
&\mathcal{F}_\varepsilon^{-1}
e_{\varepsilon\theta} ({\mathfrak A}^{\hom}_\theta+I)^{-1} \widehat{f}(\theta,\varepsilon)- (2\pi)^{-3/2} \int_{\R^3} e_\theta ({\mathfrak A}^{\hom}_\theta +I)^{-1} \widehat{f}(\theta,\varepsilon) d\theta \\[0.4em]
&=(2\pi)^{-3/2} \bigg( \int_{\varepsilon^{-1}Q'} e_\theta({\mathfrak A}^{\hom}_\theta+I)^{-1} \widehat{f}(\theta,\varepsilon) d\theta - \int_{\R^3} e_\theta({\mathfrak A}^{\hom}_\theta +I)^{-1} \widehat{f}(\theta,\varepsilon)  d\theta  \bigg)\\[0.4em]
&=-(2\pi)^{-3/2} \int_{\R^3\setminus\varepsilon^{-1}Q'} e_\theta ({\mathfrak A}^{\hom}_\theta+I)^{-1} \widehat{f}(\theta,\varepsilon)d\theta.
\end{align*}
It follows that there exists $C>0$ such that
\begin{equation*}
\begin{aligned}
\bigl\|u^\varepsilon-u^\varepsilon_{\rm hom}\bigr\|_{L^2(\R^3, d\mu^\varepsilon)}&\leq C \varepsilon\|f\|_{L^2(\R^3, d\mu^\varepsilon)} + \frac{\varepsilon^2}{\Vert(A^{\rm hom})^{-1}\Vert^{-1}\pi^2+\varepsilon^2}\bigl\|\widehat{f}(\cdot,\varepsilon)\bigr\|_{L^2(\R^3)}\\[0.3em]
&=\biggl(C \varepsilon+\frac{\varepsilon^2}{\Vert(A^{\rm hom})^{-1}\Vert^{-1}\pi^2+\varepsilon^2}\biggr)\|f\|_{L^2(\R^3, d\mu^\varepsilon)},
\end{aligned}
\end{equation*}
which implies the claim.
\end{proof}

For each $\theta\in\varepsilon^{-1}Q',$ we define
\begin{equation}\label{def_Nsum}
 N_\theta:=\widehat{N}+a_\theta,\quad a_\theta\in\C^{3\times 3},\qquad \curl N_\theta:=\curl\widehat{N},
\end{equation}
where the matrix-valued function $\widehat{N}$ solves \eqref{curl_cellpb}, 
and the matrix $a_\theta$ is chosen so that 
\begin{equation}\label{condition_NK}
\int_Q \theta \times A \bigl( \theta \times
 N_\theta(\theta\times c)\bigr)
=0\qquad \forall c\in{\mathbb R}^3. 
\end{equation}
In what follows we show that such a choice is possible. We first prove an auxiliary proposition.

\begin{prop}
\label{aux_prop}
There exists a unique matrix $a_\theta \in{\mathbb R}^{3\times 3}$ such that 
\begin{equation}
a_\theta \theta=0,\qquad\ \ \  a_\theta \zeta\cdot\theta=0\quad\ \   \forall\zeta\in\Theta^\perp:=\{\zeta\in{\mathbb R}^3: \zeta\cdot\theta=0\},
\label{a_cond}
\end{equation}
and 
\begin{equation}
\label{condition_NK_eta}
\int_Q \theta \times A \bigl( \theta \times
a_\theta\zeta\bigr)
=-\int_Q \theta \times A \bigl( \theta \times
\widehat{N}\zeta\bigr)\qquad \forall\zeta\in\Theta^\perp.
\end{equation} 
\end{prop}
\begin{proof}
For any orthogonal basis $\{e_1^\perp, e_2^\perp\}$ of $\Theta^\perp,$ the identity (\ref{condition_NK_eta}) is equivalent to a linear system for the representation of the matrix $a_\theta$ in the basis $\{\theta/|\theta|, e_1^\perp, e_2^\perp\}$ of ${\mathbb R}^3.$
This system is uniquely solvable, subject to the conditions (\ref{a_cond}), for any right-hand side if and only if the only solution to the related homogeneous system is zero. The latter is easily verified, by noticing that if 
\[
\int_Q \theta \times A \bigl( \theta \times
a_\theta\zeta\bigr)=0\qquad\forall\zeta\in\Theta^\perp, 
\]
then, in particular,
\[
\biggl(\int_QA\biggr)\bigl( \theta \times
a_\theta \zeta\bigr)\cdot (\theta\times a_\theta\zeta)=0,
\]
from which we infer, due to the fact that $A$ is positive definite, that $\theta\times a_\theta\zeta=0,$ and therefore $a_\theta\zeta=0$ by the second condition in (\ref{a_cond}). Taking into account the first condition in (\ref{a_cond}), we obtain $a_\theta=0,$ as required.
\end{proof}

Furthermore, we invoke the following simple observation. 

\begin{lem}
The following characterisation for $\Theta^\perp$ holds:
\[
\Theta^\perp=\{\theta\times c: c\in{\mathbb R}^3\}.
\]
\end{lem}
\begin{proof}
The inclusion $\{\theta\times c: c\in{\mathbb R}^3\}\subset\Theta^\perp$ is trivial. The opposite inclusion follows from the observation that 
$\theta\times c=0$ if and only if $c\in{\mathbb R}^3$ is parallel to $\theta,$ equivalently orthogonal to $\Theta^\perp.$ Therefore, for each $\zeta\in \Theta^\perp,$ the problem $\theta\times c=\zeta$ has a unique solution $c\in{\mathbb R}^3,$ as required. 
\end{proof}

Using the above lemma, we write the identity (\ref{condition_NK_eta}) in an equivalent form:
\begin{equation}
\int_Q \theta \times A \bigl( \theta \times
a_\theta(\theta\times c)\bigr)
=-\int_Q \theta \times A \bigl( \theta \times
\widehat{N}(\theta\times c)\bigr)\qquad \forall c\in{\mathbb C}^3,
\label{required_id}
\end{equation}  
which is the identity \eqref{condition_NK} we require.

Notice that that $ N_\theta$ 
is bounded in $L^2(Q,d\mu),$ uniformly in $\theta.$ Indeed, from (\ref{required_id}) we have 
\[
\int_Q A\bigl|\theta \times a_\theta(\theta\times c)\bigr|^2=-\int_Q A\bigl(\theta\times \widehat{N}(\theta\times c)\bigr) \cdot\bigl(\theta\times a_\theta(\theta\times c)\bigr).
\]
Using the positive-definiteness of $A$ and the conditions \eqref{a_cond}, it follows that 
\[
\bigl\| a_\theta (\theta\times c)\bigr\|_{L^2(Q, d\mu)}\leq\bigl\| \widehat{N}(\theta\times c)\bigr\|_{L^2(Q, d\mu)}\qquad\forall c\in{\mathbb C}^3,
\]
which, combined with another use of the first equation in (\ref{a_cond}), yields a uniform estimate for $a_\theta.$ Together with (\ref{def_Nsum}) this immediately implies the claimed $L^2$-bound for $ N_\theta.$



In order to prove Theorem \ref{main_thm}, we introduce a decomposition for the vector function $u^\varepsilon_\theta$, motivated by a formal asymptotic expansion in powers of $\varepsilon.$ Before proceeding, we note that in the formulae of this section, for all $\theta\in \varepsilon^{-1}Q'$ and $c_\theta^\varepsilon$ defined by (\ref{c_theta}),
we also consider $e_{\varepsilon\theta}c_\theta^\varepsilon$ as an element of $H^1_{\curl,\varepsilon\theta},$ by setting  
\begin{equation}
\label{curl_formula}
\curl(e_{\varepsilon\theta}c_\theta^\varepsilon)={\rm i}\varepsilon e_{\varepsilon\theta}(\theta\times c_\theta^\varepsilon),
\end{equation}
which is consistent with Definitions \ref{defiH1curl} and \ref{H1kappa_def}. Now, for each $\varepsilon>0,$ $\theta\in\varepsilon^{-1}Q',$ we write
\begin{equation}\label{u=U+z}
u^\varepsilon_\theta:= U^\varepsilon_\theta+ z^\varepsilon_\theta,
\end{equation}
where
\begin{equation}
\label{U_expansion}
U^\varepsilon_\theta:= \bigl(\overline{e}_{\varepsilon\theta}\nabla (e_{\varepsilon\theta}\Psi_{\varepsilon\theta})+I\bigr)c_\theta^\varepsilon+{\rm i}\varepsilon u^{(1)}_\theta+\varepsilon^2 R^\varepsilon_\theta,\qquad u^{(1)}_\theta:=  N_\theta(\theta\times c_\theta^\varepsilon).
\end{equation}
The coefficient $R^\varepsilon_\theta$ in front of $\varepsilon^2$ in (\ref{U_expansion})
is defined to be an element of $H^1_{\curl, \varepsilon\theta}$ solving the problem
\begin{equation}
\begin{aligned}
\overline{e}_{\varepsilon\theta}&\curl\bigl(A\curl (e_{\varepsilon\theta} R^\varepsilon_\theta)\bigr)+\varepsilon^2\int_Q R^\varepsilon_\theta\; d\mu+ \varepsilon^2 \overline{e}_{\varepsilon\theta}\nabla(e_{\varepsilon\theta} \Phi_{R^\varepsilon_\theta})\\[0.1em]
&=F-\varepsilon^{-2}\overline{e}_{\varepsilon\theta}\curl\bigl(A\curl (e_{\varepsilon\theta} c_\theta^\varepsilon)\bigr)
-{\rm i}\varepsilon^{-1}\overline{e}_{\varepsilon\theta}\curl\bigl(A\curl\bigl(e_{\varepsilon\theta} u^{(1)}_\theta
\bigr)\bigr)-\bigl(\overline{e}_{\varepsilon\theta}\nabla (e_{\varepsilon\theta}\Psi_{\varepsilon\theta})+I\bigr)c_\theta^\varepsilon 
\\[0.6em]
&=F+\theta\times A(\theta\times c_\theta^\varepsilon)+\theta\times A\curl\bigl(N_\theta(\theta\times c_\theta^\varepsilon)\bigr)
\\[0.4em]
&
\qquad\qquad+\curl\bigl(A\bigl(\theta\times u^{(1)}_\theta
\bigr) \bigr)
+{\rm i}\varepsilon\theta\times A\bigl(\theta \times u^{(1)}_\theta
\bigr)-\bigl(\overline{e}_{\varepsilon\theta}\nabla (e_{\varepsilon\theta}\Psi_{\varepsilon\theta})+I\bigr)c_\theta^\varepsilon
=:\mathcal{H}^\varepsilon_\theta,
\end{aligned}
\label{eq_R}
\end{equation}
where $(e_{\varepsilon\theta}\Phi_{R^\varepsilon_\theta}, \nabla(e_{\varepsilon\theta}\Phi_{R^\varepsilon_\theta}))\in H^1_{\varepsilon\theta}$ is defined from $R^\varepsilon_\theta$ by Proposition \ref{prop_eqPhi},
and the expression $\mathcal{H}^\varepsilon_\theta$ is treated as an element of the dual space $(H^1_{\curl})^*.$ In what follows, the value of the functional $\mathcal{H}^\varepsilon_\theta$ on $\phi\in H^1_{\curl}$ (recalling that we drop the mention of the second component $\curl\phi$) is denoted by $\langle\mathcal{H}^\varepsilon_\theta, \phi\rangle.$ 
The problem (\ref{eq_R}) is understood in the sense of the integral identity
\begin{equation}
\begin{aligned}
\int_Q A \curl(e_{\varepsilon\theta} R^\varepsilon_\theta)\cdot  {\curl (e_{\varepsilon\theta} \varphi)}&+\varepsilon^2 \int_Q R^\varepsilon_\theta\cdot {\int_Q \varphi}\\[0.5em]
&+\varepsilon^2 \int_Q {\overline{e}_{\varepsilon\theta}}\nabla \bigl(e_{\varepsilon\theta} \Phi_{R^\varepsilon_\theta}\bigr)\cdot {\varphi}= \langle \mathcal{H}^\varepsilon_\theta, \phi \rangle \qquad \forall\varphi\in[C_\#^\infty]^3,
\end{aligned}
\label{eq_R_id}
\end{equation}
where, in accordance with the above convention, $[C_\#^\infty]^3$ is treated as a subset of $H^1_{\curl}.$ The second equality in \eqref{eq_R} is verified by taking $\varphi\in [C_\#^\infty]^3$ and noticing that
\begin{align*}
\Bigl\langle F&-\varepsilon^{-2}\overline{e}_{\varepsilon\theta}\curl\bigl(A\curl (e_{\varepsilon\theta} c_\theta^\varepsilon)\bigr)
-{\rm i}\varepsilon^{-1}\overline{e}_{\varepsilon\theta}\curl\bigl(A\curl\bigl(e_{\varepsilon\theta} u^{(1)}_\theta\bigr)\bigr)-\bigl(\overline{e}_{\varepsilon\theta}\nabla (e_{\varepsilon\theta}\Psi_{\varepsilon\theta})+I\bigr)c_\theta^\varepsilon, \varphi\Bigr\rangle \\[0.7em]
&=\int_Q F\cdot {\varphi} -\int_Q A \curl\bigl(N_\theta({\rm i}\theta\times c_\theta^\varepsilon)\bigr)\cdot {({\rm i}\theta\times\varphi)} -\int_Q A ({\rm i}\theta\times c_\theta^\varepsilon)\cdot {({\rm i}\theta\times\varphi)} \\[0.4em]
&-{\rm i}\int_Q A\bigl({\rm i}\theta\times u^{(1)}_\theta\bigr)\cdot {\curl\varphi}
-{\rm i}\varepsilon \int_Q A\bigl({\rm i}\theta\times u^{(1)}_\theta\bigr)\cdot {({\rm i}\theta\times\varphi)}- {\int_Q\bigl(\overline{e}_{\varepsilon\theta}\nabla (e_{\varepsilon\theta}\Psi_{\varepsilon\theta})+I\bigr)c_\theta^\varepsilon\cdot\varphi},
\end{align*}
where we use the fact that, due to the second equality in (\ref{def_Nsum}) and by virtue of $\widehat{N}$ solving \eqref{curl_cellpb}, the function $N_\theta$ satisfies the equation
\[
\curl(A\curl N_\theta)=-\curl A,
\]
understood in the sense of the integral identity (\ref{weak_curlcellpb}) with $\widehat{N}$ replaced by $N_\theta.$
\begin{prop}\label{prop_existR}
For each $\varepsilon>0$ and $\theta \in \varepsilon^{-1} Q',$ there exists a unique solution $R^\varepsilon_\theta\in H^1_{\curl}$ to the problem \eqref{eq_R}. 
\end{prop}
\begin{proof}
We use the Helmholtz decomposition \eqref{decomposition} for the test function $\varphi$ in (\ref{eq_R_id}). As the vectors 
\[
{\overline{e}_{\varepsilon\theta}}\nabla\bigl(e_{\varepsilon\theta} \Phi_{R^\varepsilon_\theta}\bigr),\qquad \widetilde{\varphi}+\int_Q \varphi
\] 
are orthogonal in $L^2(Q, d\mu),$ the third term on the left-hand side of (\ref{eq_R_id}) can be written as 
\begin{align*}
\int_Q {\overline{e}_{\varepsilon\theta}}\nabla\bigl(e_{\varepsilon\theta}\Phi_{R^\varepsilon_\theta}\bigr)\cdot {\varphi}&=
\int_Q {\overline{e}_{\varepsilon\theta}}\nabla\bigl(e_{\varepsilon\theta}\Phi_{R^\varepsilon_\theta}\bigr)\cdot {\biggl(\widetilde{\varphi}+\int_Q \varphi+ {\overline{e}_{\varepsilon\theta}}\nabla(e_{\varepsilon\theta} \Phi_\varphi)\biggr)} =
\int_Q\nabla\bigl(e_{\varepsilon\theta}\Phi_{R^\varepsilon_\theta}\bigr)\cdot  {\nabla (e_{\varepsilon\theta}\Phi_\varphi)}.
\end{align*}
Hence the claim follows from the Lax-Millgram Theorem applied to the bilinear form
\begin{equation*}
b_{\varepsilon\theta}(u,v)=\int_Q A \curl(e_{\varepsilon\theta} u)\cdot  {\curl(e_{\varepsilon\theta} v)} +\varepsilon^2\int_Q u \cdot {\int_Q v}+ \varepsilon^2 \int_Q \nabla(e_{\varepsilon\theta} \Phi_u)\cdot  {\nabla(e_{\varepsilon\theta} \Phi_v)},\qquad u,v \in H^1_{\curl},
\end{equation*}
where $\Phi_u$, $\Phi_v$ are defined as in \eqref{eq_Phi}. Indeed the form is bounded and the coercivity is a consequence of the Poincar\'{e} inequality \eqref{curlpoincare}.
\end{proof}

In order to prove the estimates for $R^\varepsilon_\theta$ claimed in Theorem \ref{thm_estimate_R} below, we will use the Poincar\'{e} inequality \eqref{curlpoincare}. To this end, we notice that 
\begin{equation}
\bigl\langle \mathcal{H}^\varepsilon_\theta, R^\varepsilon_\theta\bigr\rangle=\bigl\langle \mathcal{H}^\varepsilon_\theta, \widetilde{R}^\varepsilon_\theta\bigr\rangle,
\label{twoH}
\end{equation}
where $\widetilde{R}^\varepsilon_\theta$ is defined as in the decomposition \eqref{decomposition}. The equation (\ref{twoH}) follows from the following two properties of $\mathcal{H}^\varepsilon_\theta$. First, we observe that by the definition of $\mathcal{H}^\varepsilon_\theta$, see the expression in the middle line of (\ref{eq_R}), one has
\begin{equation*}
\bigl\langle \mathcal{H}^\varepsilon_\theta, \overline{e}_{\varepsilon\theta}\nabla(e_{\varepsilon\theta}\varphi)\bigr\rangle=0\qquad \forall\varphi\in H^1_\#,
\end{equation*}
in view of the equation (\ref{Psi_equation}) for $\Psi_{\varepsilon\theta}$ and since $\overline{e}_{\varepsilon\theta}\divv(e_{\varepsilon\theta} F)=0$.
In particular,
\begin{equation}
\bigl\langle \mathcal{H}^\varepsilon_\theta, \overline{e}_{\varepsilon\theta}\nabla (e_{\varepsilon\theta} \Phi_u)\bigr\rangle=0
\label{H_div_free}
\end{equation}
for all functions $\Phi_u$ that solve \eqref{eq_Phi} for some $u\in H^1_\#.$ Furthermore, the functional $\mathcal{H}^\varepsilon_\theta$ vanishes on constant vector functions:
\begin{equation}
\label{H_const_free}
\bigl\langle \mathcal{H}^\varepsilon_\theta, d\bigr\rangle=0 \qquad \forall d\in \C^3.
\end{equation}
This is a consequence of the equation \eqref{homo_problem_ft} solved by $c_\theta^\varepsilon,$ taking into account the condition \eqref{condition_NK}.



\section{Proof of Theorem \ref{main_thm}: asymptotic estimate for $R_\theta^\varepsilon$ as $\varepsilon\to0$}

\label{proof_sec}



\begin{thm}
\label{thm_estimate_R}
There exists $C>0$ such that for all $\varepsilon>0,$ $\theta\in\varepsilon^{-1}Q',$ the solution $R^\varepsilon_\theta$ to the problem \eqref{eq_R} satisfies the estimates
\begin{align}
& \bigg\|R^\varepsilon_\theta-\int_Q R^\varepsilon_\theta-\overline{e}_{\varepsilon\theta}\nabla\bigl(e_{\varepsilon\theta} \Phi_{R^\varepsilon_\theta}\bigr)+\int_Q\overline{e}_{\varepsilon\theta}\nabla\bigl(e_{\varepsilon\theta} \Phi_{R^\varepsilon_\theta}\bigr)\bigg\|_{L^2(Q, d\mu)}\leq C\|F\|_{L^2(Q, d\mu)} \label{estimate_R1},\\[0.4em]
&\bigg\| \int_Q R^\varepsilon_\theta+\overline{e}_{\varepsilon\theta}\nabla\bigl(e_{\varepsilon\theta} \Phi_{R^\varepsilon_\theta}\bigr)-\int_Q\overline{e}_{\varepsilon\theta}\nabla\bigl(e_{\varepsilon\theta} \Phi_{R^\varepsilon_\theta}\bigr)\bigg\|_{L^2(Q, d\mu)}\leq C\varepsilon^{-1} \|F\|_{L^2(Q, d\mu)}.\label{estimate_R2}
\end{align}
\end{thm}
\begin{proof}
Suppose that $\phi_n\in[C_\#^\infty]^3$ converging to $R^\varepsilon_\theta$ in $L^2(Q, d\mu)$ are such that 
\[
\curl (e_{\varepsilon\theta} \phi_n)\stackrel{n\to\infty}{\longrightarrow}\curl (e_{\varepsilon\theta} R^\varepsilon_\theta)\quad {\rm in}\ \ L^2(Q, d\mu)
\] 
and use $\phi_n$ as test functions in the integral identity for \eqref{eq_R}:
\begin{equation}\label{weak_eqRvarphi}
\int_Q A \curl (e_{\varepsilon\theta} R^\varepsilon_\theta)\cdot {\curl(e_{\varepsilon\theta} \phi_n)}+\varepsilon^2 \int_Q R^\varepsilon_\theta\cdot {\int_Q\phi_n}+\varepsilon^2\int_Q \overline{e}_{\varepsilon\theta}\nabla\bigl(e_{\varepsilon\theta}\Phi_{R^\varepsilon_\theta}\bigr)\cdot {\phi_n}=\bigl\langle \mathcal{H}^\varepsilon_\theta, \phi_n\bigr\rangle.
\end{equation}
Invoking the 
the properties \eqref{H_div_free} and \eqref{H_const_free},
we write the right-hand side of (\ref{weak_eqRvarphi}) as follows:
\begin{equation*}
\bigl\langle \mathcal{H}^\varepsilon_\theta, \phi_n\bigr\rangle=\biggr\langle \mathcal{H}^\varepsilon_\theta, \phi_n - \int_Q R^\varepsilon_\theta -\overline{e}_{\varepsilon\theta}\nabla\bigl(e_{\varepsilon\theta} \Phi_{R^\varepsilon_\theta}\bigr)+\int_Q\overline{e}_{\varepsilon\theta}\nabla\bigl(e_{\varepsilon\theta} \Phi_{R^\varepsilon_\theta}\bigr)\biggr\rangle.
\end{equation*}
Furthermore, using the identity ({\it cf.} (\ref{curl_formula}))
\begin{equation}
\begin{aligned}
&\curl \bigg(\phi_n- \int_Q R^\varepsilon_\theta -\overline{e}_{\varepsilon\theta}\nabla(e_{\varepsilon\theta} \Phi_{R^\varepsilon_\theta})+\int_Q\overline{e}_{\varepsilon\theta}\nabla\bigl(e_{\varepsilon\theta} \Phi_{R^\varepsilon_\theta}\bigr)\bigg) \\[0.4em]
&\qquad\qquad=\overline{e}_{\varepsilon\theta}\bigg\{\curl \bigg(e_{\varepsilon\theta}\bigg(\phi_n- \int_Q R^\varepsilon_\theta -\overline{e}_{\varepsilon\theta}\nabla\bigl(e_{\varepsilon\theta} \Phi_{R^\varepsilon_\theta}\bigr)+\int_Q\overline{e}_{\varepsilon\theta}\nabla\bigl(e_{\varepsilon\theta} \Phi_{R^\varepsilon_\theta}\bigr)\biggr)\bigg)\\[0.4em]
&\qquad\qquad-{\rm i}\varepsilon e_{\varepsilon\theta}\theta
\times\bigg(\phi_n- \int_Q R^\varepsilon_\theta -\overline{e}_{\varepsilon\theta}\nabla\bigl(e_{\varepsilon\theta} \Phi_{R^\varepsilon_\theta}\bigr)+\int_Q\overline{e}_{\varepsilon\theta}\nabla\bigl(e_{\varepsilon\theta} \Phi_{R^\varepsilon_\theta}\bigr)\bigg)\bigg\},
\end{aligned}
\label{triangle}
\end{equation}
we rewrite \eqref{weak_eqRvarphi} as
\begin{align*}
&\int_Q A \curl (e_{\varepsilon\theta}R^\varepsilon_\theta)\cdot {\curl(e_{\varepsilon\theta} \phi_n)}+\varepsilon^2 \int_Q R^\varepsilon_\theta\cdot {\int_Q\phi_n}+\varepsilon^2\int_Q \overline{e}_{\varepsilon\theta}\nabla\bigl(e_{\varepsilon\theta}\Phi_{R^\varepsilon_\theta}\bigr)\cdot {\phi_n} \\[0.4em]
&=\int_Q \Big(F+\theta\times A(\theta\times c_\theta^\varepsilon) + \theta\times A\bigl(\curl N_\theta(\theta\times c_\theta^\varepsilon)\bigr)-\bigl(\overline{e}_{\varepsilon\theta}\nabla (e_{\varepsilon\theta}\Psi_{\varepsilon\theta})+I\bigr)c_\theta^\varepsilon\Big)\\[0.4em]
&\qquad\qquad\qquad\qquad\qquad\qquad\cdot{\bigg(\phi_n- \int_Q R^\varepsilon_\theta -\overline{e}_{\varepsilon\theta}\nabla\bigl(e_{\varepsilon\theta} \Phi_{R^\varepsilon_\theta}\bigr)+\int_Q\overline{e}_{\varepsilon\theta}\nabla\bigl(e_{\varepsilon\theta} \Phi_{R^\varepsilon_\theta}\bigr)\bigg)} \notag \\[0.4em]
&+\int_Q e_{\varepsilon\theta} A\bigl(\theta\times u^{(1)}_\theta\bigr)\cdot {\curl \bigg(e_{\varepsilon\theta} \bigg(\phi_n- \int_Q R^\varepsilon_\theta -\overline{e}_{\varepsilon\theta}\nabla\bigl(e_{\varepsilon\theta} \Phi_{R^\varepsilon_\theta}\bigr)+\int_Q\overline{e}_{\varepsilon\theta}\nabla\bigl(e_{\varepsilon\theta} \Phi_{R^\varepsilon_\theta}\bigr)\bigg)\bigg)}\notag
\end{align*}
In the last identity we pass to the limit as $n\to\infty$ and use the assumptions made about the convergence of the sequence $\{\phi_n\}.$ Applying the decomposition \eqref{decomposition} to the function $R^\varepsilon_\theta$, due to the property \eqref{mean_utilde_Phi}, 
the second term on the left-hand side of the resulting equality is
\begin{equation*}
\int_Q R^\varepsilon_\theta\cdot {\int_Q R^\varepsilon_\theta}=\bigg| \int R^\varepsilon_\theta \bigg|^2.
\end{equation*}
Furthermore, due to the orthogonality between 
\[
{\overline{e}_{\varepsilon\theta}}\nabla\bigl(e_{\varepsilon\theta} \Phi_{R^\varepsilon_\theta}\bigr),\qquad \widetilde{R}^\varepsilon_\theta+\int_Q R^\varepsilon_\theta,
\]
see Section \ref{Helmholtz_section}, the third term on the left-hand side is
\begin{align*}
\int_Q \overline{e}_{\varepsilon\theta}\nabla\bigl(e_{\varepsilon\theta}\Phi_{R^\varepsilon_\theta}\bigr)\cdot {R^\varepsilon_\theta}&=
\int_Q\bigl|\nabla \bigl(e_{\varepsilon\theta}\Phi_{R^\varepsilon_\theta}\bigr)\bigr|^2.
\end{align*}
Hence, we obtain
\begin{equation}
\label{eq_R_with_rhs}
\begin{aligned}
&\int_Q A \curl (e_{\varepsilon\theta} R^\varepsilon_\theta)\cdot {\curl(e_{\varepsilon\theta} R^\varepsilon_\theta )}+\varepsilon^2 \bigg| \int_QR^\varepsilon_\theta \bigg|^2+\varepsilon^2\int_Q\bigl|\nabla \bigl(e_{\varepsilon\theta}\Phi_{R^\varepsilon_\theta}\bigr)\bigr|^2 \\[0.3em]
&=\int_Q \Big(F+\theta\times A(\theta\times c_\theta^\varepsilon) + \theta\times A\curl\bigl(N_\theta(\theta\times c_\theta^\varepsilon)\bigr)\\[0.3em]
&\hspace{2in}-\bigl(\overline{e}_{\varepsilon\theta}\nabla (e_{\varepsilon\theta}\Psi_{\varepsilon\theta})+I\bigr)c_\theta^\varepsilon
\Big)\cdot 
 {\biggl({\widetilde{R}^\varepsilon_\theta}+\int_Q\overline{e}_{\varepsilon\theta}\nabla\bigl(e_{\varepsilon\theta} \Phi_{R^\varepsilon_\theta}\bigr)
\biggr)}\\[0.3em]
&+\int_Q e_{\varepsilon\theta} A\bigl(\theta\times u^{(1)}_\theta\bigr)\cdot {\curl \biggl(e_{\varepsilon\theta}\biggl({\widetilde{R}^\varepsilon_\theta}+\int_Q\overline{e}_{\varepsilon\theta}\nabla\bigl(e_{\varepsilon\theta} \Phi_{R^\varepsilon_\theta}\bigr)\biggr)
\bigg)},
\end{aligned}
\end{equation}
where $\widetilde{R}^\varepsilon_\theta$ is the first term in the Helmholtz decomposition (\ref{decomposition}) for $R^\varepsilon_\theta.$ Note that the last term on the right-hand side of (\ref{eq_R_with_rhs}) equals
\begin{equation}
\int_Q e_{\varepsilon\theta} A\bigl(\theta\times u^{(1)}_\theta\bigr)\cdot {\curl \bigl(e_{\varepsilon\theta}{\widetilde{R}^\varepsilon_\theta}\big)},
\label{auxR}
\end{equation}
due to the condition (\ref{condition_NK}).

In order to study the expression (\ref{auxR}), for each $\varepsilon>0,$ $\theta\in\varepsilon^{-1}Q',$ consider $\xi^\varepsilon_\theta\in H^1_{\curl}$ that solves 
\begin{equation}
\label{eq_xi}
\overline{e}_{\varepsilon\theta}\curl\bigl(A\curl(e_{\varepsilon\theta} \xi^\varepsilon_\theta)\bigr)+\varepsilon^2\int_Q \xi^\varepsilon_\theta+\varepsilon^2 \overline{e}_{\varepsilon\theta}\nabla\bigl(e_{\varepsilon\theta} \Phi_{\xi^\varepsilon_\theta}\bigr)=\overline{e}_{\varepsilon\theta}\curl\bigl(e_{\varepsilon\theta} A \bigl(\theta\times u^{(1)}_\theta\bigr)\bigr).
\end{equation}
The existence and uniqueness of such $\xi^\varepsilon_\theta$
follow from the same argument as the one used in Proposition \ref{prop_existR}.
Furthermore, using $\xi^\varepsilon_\theta$ as a test function in the integral identity for \eqref{eq_xi}, we obtain the uniform estimate
\begin{equation}\label{xi_estimate}
\bigl\| \curl(e_{\varepsilon\theta} \xi^\varepsilon_\theta)\bigr\|_{L^2(Q,d\mu)}\leq C \|F\|_{L^2(Q,d\mu)}.
\end{equation}

Next, testing  \eqref{eq_xi} with 
\[
{\widetilde{R}^\varepsilon_\theta}=R^\varepsilon_\theta- \int_Q R^\varepsilon_\theta -\overline{e}_{\varepsilon\theta}\nabla(e_{\varepsilon\theta} \Phi_{R^\varepsilon_\theta}),
\] 
we write the last term in \eqref{eq_R_with_rhs} as
\begin{equation}
\begin{aligned}
\int_Q e_{\varepsilon\theta} A\bigl(\theta\times u^{(1)}_\theta\bigr)\cdot {\curl \big(e_{\varepsilon\theta}{\widetilde{R}^\varepsilon_\theta}
\big)}&=\int_Q A\curl (e_{\varepsilon\theta} \xi^\varepsilon_\theta)\cdot {\curl \big(e_{\varepsilon\theta}{\widetilde{R}^\varepsilon_\theta}
\big)} 
\\[0.4em]
&
+\varepsilon^2 \int_Q \xi^\varepsilon_\theta\cdot {\int_Q{\widetilde{R}^\varepsilon_\theta}
 }
+\varepsilon^2 \int_Q \overline{e}_{\varepsilon\theta}\nabla\bigl(e_{\varepsilon\theta}\Phi_{\xi^\varepsilon_\theta}\bigr)\cdot {{\widetilde{R}^\varepsilon_\theta}
}. 
\end{aligned}
\label{eqma1}
\end{equation}
At the same time, we have 
\begin{equation}
\int_Q \xi^\varepsilon_\theta\cdot {\int_Q{\widetilde{R}^\varepsilon_\theta}}=\int_Q \xi^\varepsilon_\theta\cdot {\int_Q\bigg(R^\varepsilon_\theta- \int_Q R^\varepsilon_\theta -\overline{e}_{\varepsilon\theta}\nabla\bigl(e_{\varepsilon\theta} \Phi_{R^\varepsilon_\theta}\bigr)\bigg) }
=-\int_Q \xi^\varepsilon_\theta\cdot {\int_Q\overline{e}_{\varepsilon\theta}\nabla\bigl(e_{\varepsilon\theta}\Phi_{R^\varepsilon_\theta}\bigr)},
\label{eqma2}
\end{equation}
and
\begin{equation}
\begin{aligned}
\int_Q \overline{e}_{\varepsilon\theta}\nabla\bigl(e_{\varepsilon\theta}\Phi_{\xi^\varepsilon_\theta}\bigr)\cdot {{\widetilde{R}^\varepsilon_\theta}}&=\int_Q \overline{e}_{\varepsilon\theta}\nabla\bigl(e_{\varepsilon\theta}\Phi_{\xi^\varepsilon_\theta}\bigr)\cdot \biggl({\widetilde{R}^\varepsilon_\theta}+\int_Q R^\varepsilon_\theta\biggr)-\int_Q\overline{e}_{\varepsilon\theta}\nabla\bigl(e_{\varepsilon\theta}\Phi_{\xi^\varepsilon_\theta}\bigr)\cdot {\int_Q R^\varepsilon_\theta}\\[0.5em]
&=-\int_Q\overline{e}_{\varepsilon\theta}\nabla\bigl(e_{\varepsilon\theta}\Phi_{\xi^\varepsilon_\theta}\bigr)\cdot {\int_Q R^\varepsilon_\theta},
\end{aligned}
\label{eqma3}
\end{equation}
where for the second equality we use the fact that (see (\ref{eq_utilde}))
\[
\overline{e}_{\varepsilon\theta}\,{\rm div}\biggl\{e_{\varepsilon\theta}\bigg({\widetilde{R}^\varepsilon_\theta}+\int_Q R^\varepsilon_\theta
\bigg)\biggr\}=0.
\] 
Combining (\ref{eqma1}), (\ref{eqma2}), (\ref{eqma3}) yields 
\begin{align}
\label{rhs_tocorrect}
\int_Q e_{\varepsilon\theta} A\bigl(\theta\times u^{(1)}_\theta\bigr)\cdot {\curl\big(e_{\varepsilon\theta} 
{\widetilde{R}^\varepsilon_\theta}
\big)}&=\int_Q A\curl (e_{\varepsilon\theta} \xi^\varepsilon_\theta)\cdot {\curl\big(e_{\varepsilon\theta} 
{\widetilde{R}^\varepsilon_\theta}
\big)} \\[0.4em]
&
-\varepsilon^2\int_Q \xi^\varepsilon_\theta\cdot {\int_Q \overline{e}_{\varepsilon\theta}\nabla\bigl(e_{\varepsilon\theta}\Phi_{R^\varepsilon_\theta}\bigr)}-\varepsilon^2\int_Q \overline{e}_{\varepsilon\theta}\nabla\bigl(e_{\varepsilon\theta}\Phi_{\xi^\varepsilon_\theta}\bigr)\cdot {\int_Q R^\varepsilon_\theta}. \notag
\end{align}

We would like to rewrite the expression on the right-hand side of \eqref{rhs_tocorrect} using $\xi^\varepsilon_\theta$ as a test function in the integral identity \eqref{eq_R_id}. As we mentioned earlier, for a general measure $\mu$, the curl of an arbitrary function in $H^1_{\curl}$ is not uniquely defined. However, for the solution $\xi^\varepsilon_\theta$ to \eqref{eq_xi} there exists a natural choice of $\curl(e_{\varepsilon\theta}\xi^\varepsilon_\theta).$ Indeed, consider sequences $\{\phi_n\}$, 
$\{\psi_n\}\subset[C_\#^\infty]^3$ converging to $\xi^\varepsilon_\theta$ in $L^2(Q, d\mu)$, so that
\begin{equation*}
\curl (e_{\varepsilon\theta} \phi_n)\to\curl(e_{\varepsilon\theta} \xi^\varepsilon_\theta) \quad \quad \curl (e_{\varepsilon\theta} \psi_n)\to\curl(e_{\varepsilon\theta} \xi^\varepsilon_\theta)\qquad {\rm as}\ \ n\to\infty.
\end{equation*}
The difference $\curl (e_{\varepsilon\theta} \phi_n)-\curl (e_{\varepsilon\theta} \psi_n)$ converges to zero, and hence so does $\curl \phi_n-\curl \psi_n$. Henceforth we denote by $\curl \xi^\varepsilon_\theta$ the common $L^2$-limit of $\curl \phi_n$ for sequences $\{\phi_n\}\subset[C_\#^\infty]^3$ with the above properties. The unique choice of $\curl \xi^\varepsilon_\theta$ as above allows us to write
\begin{equation*}
\int_Q A\curl(e_{\varepsilon\theta} R^\varepsilon_\theta)\cdot {\curl(e_{\varepsilon\theta} \xi^\varepsilon_\theta)}+\varepsilon^2\int_Q R^\varepsilon_\theta\cdot {\int_Q \xi^\varepsilon_\theta}+\varepsilon^2\int_Q\overline{e}_{\varepsilon\theta}\nabla\bigl(e_{\varepsilon\theta}\Phi_{R^\varepsilon_\theta}\bigr)\cdot {\xi^\varepsilon_\theta}=\bigl\langle \mathcal{H}^\varepsilon_\theta,\xi^\varepsilon_\theta\bigr\rangle.
\end{equation*}
Furthermore, applying the decomposition \eqref{decomposition} to $\xi^\varepsilon_\theta$ and using $\Phi_{R^\varepsilon_\theta}\in H^1_{\#,0}$ as a test function for (\ref{eq_Phi}) with $u=\xi^\varepsilon_\theta$ (noting that, as $C^\infty_{\#,0}$ is dense in $H^1_{\#,0},$ the test functions $\phi$ in (\ref{eq_Phi_weak}) can be taken in $H^1_{\#,0}$),
we obtain
\begin{equation*}
\int_Q\overline{e}_{\varepsilon\theta}\nabla\bigl(e_{\varepsilon\theta}\Phi_{R^\varepsilon_\theta}\bigr)\cdot {\xi^\varepsilon_\theta}=
\int_Q\overline{e}_{\varepsilon\theta}\nabla\bigl(e_{\varepsilon\theta}\Phi_{R^\varepsilon_\theta}\bigr)\cdot {\overline{e}_{\varepsilon\theta}\nabla\bigl(e_{\varepsilon\theta}\Phi_{\xi^\varepsilon_\theta}\bigr)}=\int_Q\nabla\bigl(e_{\varepsilon\theta}\Phi_{R^\varepsilon_\theta}\bigr)\cdot {\nabla\bigl(e_{\varepsilon\theta}\Phi_{\xi^\varepsilon_\theta}\bigr)}.
\end{equation*}
By recalling also the properties \eqref{H_div_free}, \eqref{H_const_free} of $\mathcal{H}^\varepsilon_\theta,$ 
the above implies
\begin{equation*}
\begin{aligned}
\int_Q A\curl(e_{\varepsilon\theta} R^\varepsilon_\theta)\cdot {\curl(e_{\varepsilon\theta} \xi^\varepsilon_\theta)}&+\varepsilon^2\int_Q R^\varepsilon_\theta\cdot {\int_Q \xi^\varepsilon_\theta}+\varepsilon^2\int_Q\nabla\bigl(e_{\varepsilon\theta}\Phi_{R^\varepsilon_\theta}\bigr)\cdot {\nabla\bigl(e_{\varepsilon\theta}\Phi_{\xi^\varepsilon_\theta}\bigr)}
\\[0.4em]
&
=\bigl\langle \mathcal{H}^\varepsilon_\theta,\xi^\varepsilon_\theta\bigr\rangle=\bigl\langle \mathcal{H}^\varepsilon_\theta, {\widetilde{\xi}^\varepsilon_\theta}
\bigr\rangle, 
\end{aligned}
\end{equation*}
and therefore
\begin{equation}
	\label{eq_conj_xiR}
\int_Q A\curl(e_{\varepsilon\theta} \xi^\varepsilon_\theta)\cdot {\curl(e_{\varepsilon\theta} R^\varepsilon_\theta)}
+\varepsilon^2\int_Q \xi^\varepsilon_\theta\cdot {\int_Q R^\varepsilon_\theta}+\varepsilon^2\int_Q\nabla\bigl(e_{\varepsilon\theta}\Phi_{\xi^\varepsilon_\theta}\bigr)\cdot {\nabla\bigl(e_{\varepsilon\theta}\Phi_{R^\varepsilon_\theta}\bigr)} 
=\overline{\bigl\langle \mathcal{H}^\varepsilon_\theta,  {\widetilde{\xi}^\varepsilon_\theta}
\bigr\rangle
}.
\end{equation}
We now rewrite the equation \eqref{rhs_tocorrect} using \eqref{eq_conj_xiR}, as follows:
\begin{equation}
\begin{aligned}
&\int_Q e_{\varepsilon\theta} A\bigl(\theta\times u^{(1)}_\theta\bigr)\cdot {\curl \big(e_{\varepsilon\theta}{\widetilde{R}^\varepsilon_\theta}
\big)} 
=\overline{\bigl\langle \mathcal{H}^\varepsilon_\theta, {\widetilde{\xi}^\varepsilon_\theta}
\bigr\rangle}\\[0.4em]
&
-\biggl\{\int_Q A \curl(e_{\varepsilon\theta}\xi^\varepsilon_\theta)\cdot\curl\biggl(e_{\varepsilon\theta}\int_Q R^\varepsilon_\theta\biggr)
+\varepsilon^2\int_Q \xi^\varepsilon_\theta\cdot\int_Q R^\varepsilon_\theta+\varepsilon^2\int_Q\overline{e}_{\varepsilon\theta}
\nabla\bigl(e_{\varepsilon\theta}\Phi_{\xi^\varepsilon_\theta}\bigr)\cdot\int_Q R^\varepsilon_\theta\biggr\}
\\[0.4em]
&-\varepsilon^2\biggl\{
\int_Q \xi^\varepsilon_\theta\cdot {\int_Q \overline{e}_{\varepsilon\theta}\nabla\bigl(e_{\varepsilon\theta}\Phi_{R^\varepsilon_\theta}\bigr)}
+\int_Q\overline{e}_{\varepsilon\theta}\nabla\bigl(e_{\varepsilon\theta}\Phi_{\xi^\varepsilon_\theta}\bigr)\cdot {\overline{e}_{\varepsilon\theta}\nabla\bigl(e_{\varepsilon\theta}\Phi_{R^\varepsilon_\theta}\bigr)}\biggr\}.
\end{aligned}
\label{analysed_eq}
\end{equation}
The second term on the right-hand side of the last equation vanishes, by using the function identically equal to the  
vector
\[
\int_QR^\varepsilon_\theta
\] 
as a test function in the integral formulation for \eqref{eq_xi} and noting that
\begin{align*}
\int_Qe_{\varepsilon\theta} A\bigl(\theta\times u^{(1)}_\theta\bigr)\cdot{\curl\biggl(e_{\varepsilon\theta}\int_QR^\varepsilon_\theta\biggr)} ={\rm i}\varepsilon\overline{\int_QR^\varepsilon_\theta}\cdot \int_Q \theta\times A\bigl(\theta\times u^{(1)}_\theta\bigr)=0,
\end{align*}
in view of \eqref{curl_formula} and \eqref{condition_NK}.
The third term on the right-hand side of (\ref{analysed_eq}) also vanishes, by using $\overline{e}_{\varepsilon\theta}\nabla(e_{\varepsilon\theta}\Phi_{R^\varepsilon_\theta})$ as a test function in the integral formulation for \eqref{eq_xi} and taking advantage of the fact that when $\Phi_{R^\varepsilon_\theta}$ is approximated in $L^2(Q, d\mu)$ with smooth functions $\phi_n,$ the corresponding expressions $\curl(\nabla (e_{\varepsilon\theta}\phi_n))$ vanish for all $n.$

 Returning to \eqref{eq_R_with_rhs}--\eqref{auxR}, we thus obtain
\begin{equation}
\begin{aligned}
&\int_Q A \curl (e_{\varepsilon\theta}R^\varepsilon_\theta)\cdot {\curl(e_{\varepsilon\theta} R^\varepsilon_\theta )}+\varepsilon^2 \bigg| \int_QR^\varepsilon_\theta \bigg|^2+\varepsilon^2\int_Q\bigl|\nabla \bigl(e_{\varepsilon\theta}\Phi_{R^\varepsilon_\theta}\bigr)\bigr|^2\\[0.4em]
&=\int_Q \Big(F+\theta\times A(\theta\times c_\theta^\varepsilon) + \theta\times A\curl\bigl(N_\theta(\theta\times c_\theta^\varepsilon)\bigr)\\[0.4em] &\hspace{10em}-\bigl(\overline{e}_{\varepsilon\theta}\nabla (e_{\varepsilon\theta}\Psi_{\varepsilon\theta})+I\bigr)c_\theta^\varepsilon
\Big)\cdot 
\biggl({\widetilde{R}^\varepsilon_\theta}+\int_Q\overline{e}_{\varepsilon\theta}\nabla\bigl(e_{\varepsilon\theta} \Phi_{R^\varepsilon_\theta}\bigr)\biggr)
+\overline{\bigl\langle \mathcal{H}^\varepsilon_\theta, {\widetilde{\xi}^\varepsilon_\theta}.
\bigr\rangle}
\end{aligned}
\label{eq_Rfinale}
\end{equation}

To complete setting the stage for the estimates (\ref{estimate_R1}), (\ref{estimate_R2}), it remains to estimate the second term on the right-hand side of (\ref{eq_Rfinale}) by the $L^2$ norm of the function $F$ in (\ref{eq_R}), which we do next.
\begin{lem}
\label{est_lemma}
The last term on the right-hand side of \eqref{eq_Rfinale} is bounded uniformly in $\varepsilon$ and $\theta$:
\begin{equation*}
\bigl|\bigl\langle \mathcal{H}^\varepsilon_\theta, {\widetilde{\xi}^\varepsilon_\theta}
\bigr\rangle\bigr|\leq C\|F\|_{L^2(Q,d\mu)}, \quad C>0.
\end{equation*}
\end{lem}
\begin{proof}
It follows by the definition of $\mathcal{H}^\varepsilon_\theta,$ see \eqref{eq_R}, that
\begin{align*}
\bigl\langle \mathcal{H}^\varepsilon_\theta, {\widetilde{\xi}^\varepsilon_\theta}
\bigr\rangle &=\int_Q \Big(F+\theta\times A (\theta\times c_\theta^\varepsilon)+\theta\times A\bigl(\curl N_\theta(\theta\times c_\theta^\varepsilon)\bigr) 
\\[0.3em]
&\hspace{7em}-\bigl(\overline{e}_{\varepsilon\theta}\nabla (e_{\varepsilon\theta}\Psi_{\varepsilon\theta})+I\bigr)c_\theta^\varepsilon +{\rm i}\varepsilon\theta\times A\bigl(\theta\times u^{(1)}_\theta\bigr)\Big)\cdot {{\widetilde{\xi}^\varepsilon_\theta}
}\\[0.3em]
&+\int_Q A\bigl(\theta\times u^{(1)}_\theta\bigr) \cdot  {\curl {\widetilde{\xi}^\varepsilon_\theta}. 
} 
\end{align*}
Recalling the formula \eqref{curl_formula}, we write ({\it cf.} \eqref{triangle})
\begin{equation*}
{\curl {\widetilde{\xi}^\varepsilon_\theta}
}
=\overline{e}_{\varepsilon\theta}{\curl \bigl(e_{\varepsilon\theta}  {\widetilde{\xi}^\varepsilon_\theta}
\bigr)}-{\rm i}\varepsilon\theta\times{{\widetilde{\xi}^\varepsilon_\theta}
},
\end{equation*}
and thus
\begin{equation}
\begin{aligned}
\bigl\langle \mathcal{H}^\varepsilon_\theta, {\widetilde{\xi}^\varepsilon_\theta}
\bigr\rangle&=\int_Q \Big( F+\theta\times A (\theta\times c_\theta^\varepsilon)+\theta\times A\bigl(\curl N_\theta(\theta\times c_\theta^\varepsilon)\bigr)-\bigl(\overline{e}_{\varepsilon\theta}\nabla (e_{\varepsilon\theta}\Psi_{\varepsilon\theta})+I\bigr)c_\theta^\varepsilon\Big) 
 \cdot  {{\widetilde{\xi}^\varepsilon_\theta}
}\\[0.5em]
&+\int_Q e_{\varepsilon\theta} A\bigl( \theta\times u^{(1)}_\theta\bigr) \cdot  {\curl(e_{\varepsilon\theta} \xi^\varepsilon_\theta)},
\end{aligned}
\label{estimateHxi}
\end{equation}
since
\begin{equation*}
\int_Q e_{\varepsilon\theta} A \bigl(\theta\times u^{(1)}_\theta\bigr)\cdot {\curl\biggl(e_{\varepsilon\theta}\int_Q \xi^\varepsilon_\theta\biggr)}={\rm i}\varepsilon\int_Q \theta\times A\bigl(\theta\times u^{(1)}_\theta\bigr)\cdot {\int_Q\xi^\varepsilon_\theta}=0,
\end{equation*}
by the condition \eqref{condition_NK}.

Applying the H\"{o}lder inequality to the right-hand side of the equation \eqref{estimateHxi}, using the Poincar\'{e} inequality \eqref{curlpoincare} for $\xi^\varepsilon_\theta$, and taking into the account the estimate \eqref{xi_estimate} yields the claim.
\end{proof}
Combining Lemma \ref{est_lemma}, the Poincar\'{e} inequality \eqref{curlpoincare} with $u=R^\varepsilon_\theta,$ and the H\"{o}lder inequality for the first term on the right-hand side of the equation \eqref{eq_Rfinale}, we obtain the uniform bound
\begin{equation}\label{estimate_curlekR}
\bigl\|\curl(e_{\varepsilon\theta} R^\varepsilon_\theta)\bigr\|_{L^2(Q,d\mu)}\leq C \|F\|_{L^2(Q,d\mu)}.
\end{equation}
Finally, the estimate \eqref{estimate_curlekR} combined with \eqref{curlpoincare}, applied to $u=R^\varepsilon_\theta$ again, implies the estimate \eqref{estimate_R1}. The same bound, Lemma \ref{est_lemma}, and the equation \eqref{eq_Rfinale} imply the estimate \eqref{estimate_R2}.
\end{proof}

\begin{cor} \label{cor_U-ctheta}
There exists $C>0$ such that the following estimate holds for all $\varepsilon$, $\theta$ and $F:$
\begin{equation*}
\Bigl\|U^\varepsilon_\theta-\bigl(\overline{e}_{\varepsilon\theta}\nabla (e_{\varepsilon\theta}\Psi_{\varepsilon\theta})+I\bigr)c_\theta^\varepsilon\Bigr\|_{L^2(Q, d\mu)}\leq C \varepsilon\|F\|_{L^2(Q, d\mu)}.
\end{equation*}
\end{cor}

\subsection{Conclusion of proof: 
convergence estimate for $z_\theta^\varepsilon$}
\label{section_final}
\begin{prop}\label{prop_z}
There exists $C>0$ such that the function $z_\theta^\varepsilon$ in
 \eqref{u=U+z} satisfies the estimates
\begin{align}
&\|z_\theta^\varepsilon\|_{L^2(Q,d\mu)}\leq C\varepsilon \|F\|_{L^2(Q, d\mu)},\label{estimate_z}\\[0.4em] 
&\bigl\|\curl (e_{\varepsilon\theta} z^\varepsilon_\theta)\bigr\|_{L^2(Q,d\mu)}\leq C\varepsilon^2 \|F\|_{L^2(Q, d\mu)},\label{curl_z}
\end{align} 
for all $\varepsilon>0,$ $\theta\in\varepsilon^{-1} Q',$ $F\in L^2(Q, d\mu).$
\end{prop}
\begin{proof}
The function $z_\theta^\varepsilon\in H^1_{\curl},$ see (\ref{u=U+z}), (\ref{U_expansion}), solves the problem
\begin{equation}
\label{eq_z}
\varepsilon^{-2} \overline{e}_{\varepsilon\theta} \curl\bigl(A \curl (e_{\varepsilon\theta} z^\varepsilon_\theta)\bigr)+z_\theta^\varepsilon=
-{\rm i}\varepsilon u^{(1)}_\theta
-\varepsilon^2{\widetilde{R}^\varepsilon_\theta},
\end{equation}
understood in the weak sense. Using $z_\theta^\varepsilon$ as a test function in the integral formulation of (\ref{eq_z}), we obtain
\begin{equation}
\varepsilon^{-2}\int_Q A\curl(e_{\varepsilon\theta} z_\theta^\varepsilon)\cdot{\curl(e_{\varepsilon\theta} z_\theta^\varepsilon)}+ \int_Q|z_\theta^\varepsilon|^2
=-{\rm i}\varepsilon \int_Q u^{(1)}_\theta\cdot{z_\theta^\varepsilon}-\varepsilon^2\int_Q {\widetilde{R}^\varepsilon_\theta}
\cdot {z_\theta^\varepsilon}.
\label{rhs_eq}
\end{equation}

Using the estimate 
\[
\bigl\|R^\varepsilon_\theta\bigr\|\leq C\varepsilon^{-1} \|F\|_{L^2(Q, d\mu)},
\]
which follows from (\ref{estimate_R1}), (\ref{estimate_R2}), the elliptic estimate for the equation
\[
{\overline{e}_{\varepsilon\theta}}\triangle\bigl(e_{\varepsilon\theta} \Phi_{R^\varepsilon_\theta}\bigr) = {\overline{e}_{\varepsilon\theta}}\divv\bigl(e_{\varepsilon\theta} R^\varepsilon_\theta\bigr)
\]
and then the observation that 
\[
\widetilde{R}^\varepsilon_\theta=\biggl\{R^\varepsilon_\theta-\int_Q R^\varepsilon_\theta-\overline{e}_{\varepsilon\theta}\nabla\bigl(e_{\varepsilon\theta} \Phi_{R^\varepsilon_\theta}\bigr)+\int_Q\overline{e}_{\varepsilon\theta}\nabla\bigl(e_{\varepsilon\theta} \Phi_{R^\varepsilon_\theta}\bigr)\biggr\}-\int_Q\overline{e}_{\varepsilon\theta}\nabla\bigl(e_{\varepsilon\theta} \Phi_{R^\varepsilon_\theta}\bigr),
\]
we infer from (\ref{estimate_R1}) that  
\begin{equation}
\bigl\|\widetilde{R}^\varepsilon_\theta\bigr\|\leq C\varepsilon^{-1} \|F\|_{L^2(Q, d\mu)}.
\label{R_tilde}
\end{equation}
Now, applying the H\"{o}lder inequality to each term on the right-hand side of (\ref{rhs_eq}), 
the formula \eqref{c_theta}, and finally the estimate (\ref{R_tilde}),
we obtain \eqref{estimate_z}.

The estimate (\ref{curl_z}) follows immediately from (\ref{rhs_eq}), by using the uniform positive-definiteness of the matrix $A,$ 
applying once again the H\"{o}lder inequality to each term on its right-hand side, and using the estimate (\ref{estimate_z}) we have just obtained as well as the estimates for $u^{(1)}_\theta,$ $\widetilde{R}^\varepsilon_\theta$ that we derived in our proof of (\ref{estimate_z}).
\end{proof}
Combining Corollary \ref{cor_U-ctheta} and Proposition \ref{prop_z}, we obtain \eqref{main_stat}, since
\begin{equation*}
\Bigl\|u^\varepsilon_\theta-\bigl(\overline{e}_{\varepsilon\theta}\nabla (e_{\varepsilon\theta}\Psi_{\varepsilon\theta})+I\bigr)c_\theta^\varepsilon\Bigr\|_{L^2(Q,d\mu)}\leq \|z^\varepsilon_\theta\|_{L^2(Q, d\mu)}+\Bigl\|U^\varepsilon_\theta-\bigl(\overline{e}_{\varepsilon\theta}\nabla (e_{\varepsilon\theta}\Psi_{\varepsilon\theta})+I\bigr)c_\theta^\varepsilon\Bigr\|_{L^2(Q, d\mu)},
\end{equation*}
which concludes the proof of Theorem \ref{main_thm}.

\section{Estimates for the electric field and displacement}

\label{electric_sec}

In what follows, we refer to the non-dimensional version of the Maxwell system in the frequency domain ({\it cf.} the system (\ref{maxwell_timeharmoperator}) with $z=1$):
\begin{equation}
\left\{\begin{array}{ll}
\dfrac{\nu_0}{\nu}\curl E^\varepsilon+{\rm i}H^\varepsilon=0,\\[1em]
\dfrac{\eta_0}{\eta}\curl H^\varepsilon-{\rm i}E^\varepsilon=\dfrac{\eta_0}{\eta}J^\varepsilon,
\end{array}\right.
\label{nondim}
\end{equation}
where $\eta_0,$ $\nu_0$ are the permittivity and permeability of vacuum,
$\eta=\eta(\cdot/\varepsilon),$ $\nu=\nu(\cdot/\varepsilon)$ are the ($\varepsilon$-periodic) permittivity and permeability of the medium,
and $E^\varepsilon, H^\varepsilon, J^\varepsilon$ are dimensionless quantities, which we henceforth refer to as the ``magnetic field", ``electric field", and ``current density".

As mentioned in the introduction, Theorem \ref{main_thm} concerns the Maxwell system in the non-magnetic case $(\nu=\nu_0)$
and without external currents: formally replacing $(0, (\eta_0/\eta)J^\varepsilon)$ on the right-hand side of (\ref{nondim}) by $(-{\rm i}f^\varepsilon, 0),$ we obtain 
\begin{equation}\label{Max_system}
\begin{cases}
\curl E^\varepsilon
+{\rm i}H^\varepsilon=-{\rm i}f^\varepsilon,\\[0.3em]
A(\cdot/\varepsilon)\curl H^\varepsilon-{\rm i}E^\varepsilon=0,
\end{cases}
\end{equation}
where the coefficient matrix $A$ stands for the inverse relative dielectric permittivity 
$\eta_0/\eta.$
By eliminating $E^\varepsilon$ from (\ref{Max_system}), we obtain ({\it cf.} (\ref{curleq_u}))
\begin{equation}
\curl\bigl(A(\cdot/\varepsilon)\curl H^\varepsilon\bigr)-H^\varepsilon=f^\varepsilon.
\label{curl_form}
\end{equation}
The equation (\ref{curl_form}) describes the actual physical behaviour of the magnetic field, and is therefore set ``on the spectrum'', so $\lambda=-1$ in the ``resolvent" formulation
\begin{equation}
\curl\bigl(A(\cdot/\varepsilon)\curl u^\varepsilon\bigr)+\lambda u^\varepsilon=f^\varepsilon.
\label{lambda_res}
\end{equation}
where the solution $u^\varepsilon$ represents the magnetic field $H^\varepsilon.$
In order to study the above problem quantitatively (aiming eventually at the behaviour of original time-dependent system), we allow the parameter $\lambda$ to take any complex values in the complement of the negative half-line, and as our estimates are valid uniformly in any compact subset as long as they are established for one specific value of $\lambda,$ we set $\lambda=1$ to obtain the equation (\ref{curleq_u}). This new resolvent formulation corresponds to the following analogue of (\ref{Max_system}) ``away from the spectrum'':
\begin{equation}
\label{Max_system1}
\begin{cases}
\curl E^\varepsilon
+H^\varepsilon=f^\varepsilon,\\[0.3em]
A(\cdot/\varepsilon)\curl H^\varepsilon-E^\varepsilon=0.
\end{cases}
\end{equation}
In our discussion of (\ref{Max_system1}), we continue referring to $E^\varepsilon,$ $H^\varepsilon$ as the electric and magnetic field, in line with the existing literature on the subject of norm-resolvent estimates in homogenisation, see for example \cite{BS07}. 

In view of the above discussion, the estimate \eqref{corollary_estimate} holds for the magnetic field $H^\varepsilon$ and magnetic induction $B^\varepsilon$ (which coincide in the context of the formulation (\ref{nondim})).
We complete the analysis by establishing an operator-norm estimate for 
the electric displacement
$D^\varepsilon=A(\cdot/\varepsilon)^{-1}E^\varepsilon.$

As in Section \ref{Floquet_section}, starting from \eqref{Max_system} one obtains the following Gelfand-transformed system:
\begin{equation}\label{ft_system}
\begin{cases}\varepsilon^{-1}{\overline{e}_{\varepsilon\theta}}\curl(A e_{\varepsilon\theta} D^\varepsilon_\theta)+ H^\varepsilon_\theta=F,\\[0.35em]
\varepsilon^{-1} {\overline{e}_{\varepsilon\theta}}\curl(e_{\varepsilon\theta} H^\varepsilon_\theta) = D^\varepsilon_\theta,
\end{cases}
\end{equation}
where $H^\varepsilon_\theta$ coincides with $u^\varepsilon_\theta$ defined in \eqref{ft_problem}, and $D^\varepsilon_\theta := {\overline{e}_{\varepsilon\theta}}
\mathcal{F}_\varepsilon D^\varepsilon$. Recall that $F$ is $\divv_{\varepsilon\theta}$-free, and so are the fields $D^\varepsilon_\theta,$ $H^\varepsilon_\theta.$
Substituting \eqref{U_expansion} into the second line of \eqref{ft_system} yields
\begin{align*}
D^\varepsilon_\theta &= \varepsilon^{-1}{\overline{e}_{\varepsilon\theta}}\curl e_{\varepsilon\theta}\Bigl(\bigl(\overline{e}_{\varepsilon\theta}\nabla (e_{\varepsilon\theta}\Psi_{\varepsilon\theta})+I\bigr)c_\theta^\varepsilon+\varepsilon N_\theta ({\rm i}\theta\times c_\theta^\varepsilon)+\varepsilon^2 R^\varepsilon_\theta+z^\varepsilon_\theta\Bigr)
\\[0.4em]
&
=(\curl\widehat{N}+I) ({\rm i}\theta\times c_\theta^\varepsilon) +\varepsilon \big({\rm i}\theta\times N_\theta({\rm i}\theta\times c_\theta^\varepsilon) +{\overline{e}_{\varepsilon\theta}}\curl (e_{\varepsilon\theta} R^\varepsilon_\theta) \big)+\varepsilon^{-1}{\overline{e}_{\varepsilon\theta}}\curl (e_{\varepsilon\theta} z^\varepsilon_\theta),
\end{align*}
where $c_\theta^\varepsilon$ solves \eqref{c_theta}, $\widehat{N}$ solves (\ref{curl_cellpb}), and $N_\theta$ is defined in \eqref{def_Nsum}, $R^\varepsilon_\theta$ is the solution of \eqref{eq_R}, and $z^\varepsilon_\theta$ solves (\ref{eq_z}). As a consequence of the estimates \eqref{estimate_curlekR}, \eqref{curl_z}, we obtain the following result.
\begin{thm}
\label{thm_ED}
There exists a constant $C>0$ independent of $\theta$, $\varepsilon$ and $F$ such that
\begin{align*}
\bigl\| D^\varepsilon_\theta - (\curl\widehat{N}+I)({\rm i}\theta\times c_\theta^\varepsilon)\bigr\|_{L^2(Q, d\mu)} &\leq \varepsilon C \|F\|_{L^2(Q, d\mu)}.
\end{align*}
\end{thm}

Similarly to Theorem \ref{corollary_main_thm}, we then obtain the following estimates for the original fields on ${\mathbb R}^3.$
\begin{thm}
\label{cor_DE}
There exists $C>0,$ independent of $\varepsilon$ and the choice of $f^\varepsilon\in L^2(\R^3, d\mu^\varepsilon),$ such that
\begin{equation}
	\bigl\| D^\varepsilon-D_{\rm hom}^\varepsilon\bigr\|_{L^2(\R^3, d\mu^\varepsilon)}
\leq \varepsilon C \|f^\varepsilon\|_{L^2(\R^3, d\mu^\varepsilon)},
\label{D_est}
\end{equation}
where
\begin{equation*}
\begin{aligned}
D_{\rm hom}^\varepsilon(x):=\dfrac{1}{(2\pi)^{3/2}}\biggl\{\curl \widehat{N}\biggl(\frac{x}{\varepsilon}\biggr)+I\biggr\} 
\curl\int_{\R^3}\int_{\R^3}e_\theta(x-y)
\big(\mathfrak{A}_\theta^{\rm hom}+M^{\rm hom}_{\varepsilon\theta}\big)^{-1}f^\varepsilon(y)  d\mu^\varepsilon(y)d\theta,
\\[-0.1em]
\hspace{30em}x\in{\mathbb R}^3.
\end{aligned}
\end{equation*}
\end{thm}

\begin{rmk}
	Contrary to the estimate \eqref{corollary_estimate} for the magnetic field $u^\varepsilon$ (and, hence, magnetic induction), the estimate \eqref{D_est} for the electric displacement $D^\varepsilon$ (and hence a similar estimate for the field $E^\varepsilon=A(\cdot/\varepsilon)D^\varepsilon$)
	contains an oscillatory term,
	corresponding to the so-called ``order-zero corrector" \cite{Sus05}, \cite{Sus08}.
	In the case when $\mu$ is the Lebesgue measure, the matrix $\curl\widehat{N}$ (in the case of the electric field, the matrix  $A(\curl\widehat{N}+I)(A^{\rm hom})^{-1}-I$) 
	has zero mean over $Q.$ This recovers the classical result concerning the weak convergence $D^\varepsilon-\widetilde{D}_{\rm hom}^\varepsilon\rightharpoonup0$ (similarly $E^\varepsilon-\widetilde{E}_{\rm hom}^\varepsilon\rightharpoonup0$ for  $\widetilde{E}_{\rm hom}^\varepsilon:=A^{\rm hom}\widetilde{D}_{\rm hom}^\varepsilon$) as $\varepsilon\to0$, where 
	\begin{equation*}
		\begin{cases}
		\curl\bigl(A^{\rm hom}\widetilde{D}_{\rm hom}^\varepsilon\bigr)+u_{\rm hom}^\varepsilon=f^\varepsilon,
		\\[0.25em]
		\curl u_{\rm hom}^\varepsilon-\widetilde{D}_{\rm hom}^\varepsilon=0.
		\end{cases}
	\end{equation*} 
\end{rmk}

\section{Further developments of the method}

{\bf 1.} One physically relevant setting is the one where the magnetic permeability is still unitary, but the system has external currents. In this case, it is convenient to write it in the form ({\it cf.} (\ref{Max_system1}))
\begin{equation}\label{Maxwell_current}
\begin{cases}
\curl\bigl(A(\cdot/\varepsilon)D^\varepsilon\bigr)+H^\varepsilon=0,\\[0.3em]
\curl H^\varepsilon-D^\varepsilon=g^\varepsilon,
\end{cases}
\end{equation}
where $g^\varepsilon$ represents the divergence-free current density, and the magnetic field and $H^\varepsilon$ and electric displacement $D^\varepsilon$ are sought to be divergence-free. This is, in some sense, an intermediate case between the one analysed in the present paper and the general case with non-unitary magnetic permeability.

Following \cite{Sus08}, it is convenient to set  ${A}^{1/2}D^\varepsilon=: \mathcal{D}^\varepsilon$ in \eqref{Maxwell_current}, so the system is equivalent to 
\begin{equation}\label{eq_maxD}
{A}^{1/2}\curl\curl({A}^{1/2}\mathcal{D}^\varepsilon)+\mathcal{D}^\varepsilon= -\widetilde{A}^{1/2}g^\varepsilon,\qquad \divv({A}^{-1/2}\mathcal{D}^\varepsilon)=0. 
\end{equation}
Furthermore, as in \cite{Sus08},
we define an operator of the problem \eqref{eq_maxD} by the quadratic form
\begin{equation*}
{\mathfrak d}_\varepsilon(w,w)=\int_{\R^3}\Bigl(\curl ({A}^{1/2}w)\cdot  {\curl ({A}^{1/2}w)}+\bigl|\divv ({A}^{-1/2}w)\bigr|^2\Bigr)d\mu^\varepsilon,
\end{equation*}
with domain
\begin{equation*}
{\rm dom}({\mathfrak d}_\varepsilon)=\bigl\{w\in L^2(\R^3, d\mu^\varepsilon): \curl({A}^{1/2}w)\in L^2(\R^3, d\mu^\varepsilon), 
\,\divv({A}^{-1/2}w)\in L^2(\R^3, d\mu^\varepsilon)\bigr\}.
\end{equation*}

\noindent{\bf 2.} In the general setting with variable permittivity and permeability, the Maxwell system has the form
\begin{equation}
\begin{cases}
\curl\bigl(A(\cdot/\varepsilon)D^\varepsilon\bigr)+B^\varepsilon=f^\varepsilon,\\[0.35em]
\curl\bigl(\widetilde{A}(\cdot/\varepsilon)B^\varepsilon\bigr)-D^\varepsilon=g^\varepsilon,
\end{cases}
\label{general_eq}
\end{equation}
with periodic matrices $A^{-1}$ (inverse of the relative permittivity), $\widetilde{A}^{-1}$ (inverse of the relative permeability). In the equation (\ref{general_eq}), $B^\varepsilon$ represents magnetic induction, and $f^\varepsilon,$ $g^\varepsilon$ are divergence-free $L^2$ functions. In what follows, we write $A,$ $\widetilde{A}$ in place of $A(\cdot/\varepsilon),$ $\widetilde{A}(\cdot/\varepsilon),$ respectively, and without loss of generality assume that\footnote{This also corresponds to the physical form of the Maxwell system, where $f^\varepsilon$ is the current density and $g^\varepsilon=0.$} $g^\varepsilon=0.$

As in the previous remark, following \cite{Sus08} and labelling $\widetilde{A}^{1/2}B^\varepsilon=: \mathcal{B}^\varepsilon,$ the system is written equivalently as
\begin{equation}
\widetilde{A}^{1/2}\curl \big(A \curl(\widetilde{A}^{1/2}\mathcal{B}^\varepsilon)\big)+\mathcal{B}^\varepsilon= -\widetilde{A}^{1/2}f^\varepsilon,\qquad \divv\bigl(\widetilde{A}^{-1/2}\mathcal{B}^\varepsilon\bigr)=0. 
\label{AB_prob}
\end{equation}
The operator of the problem (\ref{AB_prob}) 
is defined by the quadratic form
\begin{equation*}
{\mathfrak b}_\varepsilon(u,u)=\int_{\R^3} \Bigl( A\curl (\widetilde{A}^{1/2}w)\cdot  {\curl (\widetilde{A}^{1/2}w)}+\bigl|\divv (\widetilde{A}^{-1/2}w)\bigr|^2\Bigr)d\mu^\varepsilon,
\end{equation*}
with domain
\begin{equation*}
{\rm dom}({\mathfrak b}_\varepsilon)=\bigl\{u\in L^2(\R^3, d\mu^\varepsilon): \curl(\widetilde{A}^{1/2}u)\in L^2(\R^3, d\mu^\varepsilon), \,\divv(\widetilde{A}^{-1/2}u)\in L^2(\R^3, d\mu^\varepsilon)\bigr\}.
\end{equation*}

\noindent{\bf 3.} In both above cases, we represent $L^2(\R^3, d\mu^\varepsilon)$ as an orthogonal sum of the ``solenoidal" and ``potential" subspaces
\begin{align*}
L^2_{\rm sol}({\mathbb R}^3, d\mu^\varepsilon)&=\bigl\{u\in L^2(\R^3, d\mu^\varepsilon): \divv\bigl(A^{-1/2}u\bigr)=0\bigr\},
\\[0.4em]
L^2_{\rm pot}({\mathbb R}^3, d\mu^\varepsilon)&=\bigr\{A^{-1/2}\nabla v: \nabla v\in L^2(\R^3, d\mu^\varepsilon)\bigr\}
\end{align*}
and prove an appropriate version of the Helmholtz decomposition and Poincar\'{e} inequality ({\it cf.} Section \ref{Helmholtz_section})
for quasiperiodic functions, following an application of the Floquet transform to $L^2({\mathbb R}^3, d\mu^\varepsilon),$ as in Section \ref{Floquet_section}. This allows us to pursue an asymptotic procedure similar to the one we describe in Sections \ref{section_as}--\ref{proof_sec} of the present paper. We shall present the related argument in a future publication.

\renewcommand{\theequation}{A.\arabic{equation}}
\renewcommand{\thesubsection}{A.\arabic{subsection}}
\setcounter{equation}{0}


\section*{Appendix A: Non-dimensionalisation of the Maxwell system}

The system of Maxwell equations describing electromagnetic phenomena in ${\mathbb R}^3$ 
is given by (see {\it e.g.} \cite{Cessenat}, \cite{Jackson})
\begin{equation*}
\begin{cases}
\curl \mathbb{E}=-\partial_t \mathbb{B},\\[0.25em]
\curl \mathbb{H}=\partial_t \mathbb{D}+\mathcal{J},\\[0.25em]
\divv \mathbb{D}=\rho, \quad \divv \mathbb{B}=0,
\end{cases}
\end{equation*}
where $\mathbb{E}$ and $\mathbb{H}$ are the electric and magnetic fields, $\mathbb{D}$ is the electric displacement and $\mathbb{B}$ is the magnetic induction, $\mathcal{J}$ is the current density and $\rho$ the charge density. In the present work we assume that $\rho=0.$ 

The fields $\mathbb{E}$ and $\mathbb{D}$, $\mathbb{B}$ and $\mathbb{H}$ are linked by the constitutive relations
\begin{equation}
\label{constitutive_relations}
\mathbb{D}=\widehat{\eta}\stackrel{t}{*}\mathbb{E},\quad \mathbb{B}=\widehat\nu\stackrel{t}{*}\mathbb{H},
\end{equation}
where the convolutions are taken with respect the time $t.$ In (\ref{constitutive_relations}), the $\widehat{\eta}(x,t)$ is the dielectric permittivity and $\widehat{\nu}(x,t)$ is the magnetic permeability, which are ingeneral time-dependent. We consider the case when (\ref{constitutive_relations})  are local in time
so $\widehat{\eta}(x,t)=\eta(x)\delta(t)$ and $\widehat{\nu}(x,t)=\nu(x)\delta(t)$. The resulting Maxwell system
 is
\begin{equation}
\begin{cases}
\curl \mathbb{E}=-\nu\partial_t \mathbb{H},\\[0.25em]
\curl \mathbb{H}=\eta\partial_t \mathbb{E}+\mathcal{J},\\[0.25em]
\divv \mathbb{E}=0, \quad \divv \mathbb{H}=0.
\end{cases}
\label{Max}
\end{equation}

Formally applying the Fourier transform in time to the system (\ref{Max}) leads to its version with harmonic time dependence of 
$\mathcal{J}$ $\mathbb{E},$ $\mathbb{H}:$ 
\begin{equation*}
\mathcal{J}={\rm e}^{{\rm i}\omega t}J,\quad \mathbb{E}={\rm e}^{{\rm i}\omega t}E, \quad \mathbb{H}={\rm e}^{{\rm i}\omega t}H,
\end{equation*}
where $\omega$ plays the role of the frequency. For a given current density amplitude $J,$ the vector $(E, H)$ satisfies
\begin{equation}
\label{maxwell_timeharm}
\begin{cases}
\nu^{-1} \curl E= -{\rm i}\omega H,\\[0.25em]
\eta^{-1} \curl H = {\rm i}\omega E+\eta^{-1}J,\\[0.25em]
\divv E=0, \quad \divv H=0.
\end{cases}
\end{equation}

In order to study the system of Maxwell equations from the mathematical point of view, we write it in a dimensionless way. Following the idea developed in \cite[Chapter 1]{bar2003} we define
\begin{equation}\label{dimension_HE}
H= \phi \widetilde{H}, \quad E= \psi \widetilde{E},
\end{equation}
where $\phi$, $\psi$ are some fixed values with the same dimensions as the magnetic and electric fields, and $\widetilde{H}$, $\widetilde{E}$ are the corresponding dimensionless vectors representing the magnetic and electric fields.
Starting from \eqref{maxwell_timeharm}, it is sufficient to do the dimensional analysis for the homogeneous system, where $J=0.$ Using \eqref{dimension_HE}, the second line 
of \eqref{maxwell_timeharm} takes the form
\begin{equation}\label{eq_EH_dimension}
\eta^{-1} \frac{\phi}{\psi} \curl  \widetilde{H} ={\rm i}\omega \widetilde{E}.
\end{equation}
Denote by $\eta_0$ and $\nu_0$ are the permittivity and permeability of the vacuum. Combining the first two equations in (\ref{maxwell_timeharm}), we infer that the dimensions of $\sqrt{\eta_0/\nu_0}$ are the same as those of $\phi/\psi.$ Henceforth we choose $\phi,$ $\psi$ so that $\phi/\psi=\sqrt{\eta_0/\nu_0}.$
Multiplying both sides of \eqref{eq_EH_dimension} by $\sqrt{\eta_0}$ one has
\begin{equation*}
\bigg(\frac{\eta}{\eta_0}\bigg)^{-1} \curl  \widetilde{H}={\rm i}\omega \sqrt{\nu_0 \eta_0} \widetilde{E}.
\end{equation*}

Assuming a periodic spatial dependence of $\eta,$ $\nu,$ we write the first equation in \eqref{maxwell_timeharm} as 
\begin{equation}
\label{eq_EH_dimension_1}
 \biggl\{\frac{\eta}{\eta_0}\!\biggl(\frac{x}{d}\biggr)\biggr\}^{-1}\curl_x\widetilde{H}={\rm i}\sqrt{\nu_0 \eta_0} \widetilde{E},
\end{equation}
where $d$ is the period. 
Note that
$\omega=2\pi c/\lambda,$
where $c$ and $\lambda$ are the wave speed and wavelength in vacuum. 
Introduce the non-dimensional parameter $\widetilde{x}= 2\pi x/\lambda$, so \eqref{eq_EH_dimension_1} becomes
\[
\frac{2\pi}{\lambda}\bigg\{\frac{\eta}{\eta_0}\bigg(\frac{\widetilde{x}}{2\pi d/ \lambda}\bigg)\bigg\}^{-1}\curl_{\widetilde{x}}  \widetilde{H}={\rm i}\frac{2\pi c_0}{\lambda}\sqrt{\nu_0 \eta_0} \widetilde{E}.
\]
Noting that $c\sqrt{\nu_0 \eta_0}=1$ and rescaling $y=\widetilde{x}/\varepsilon,$ where $\varepsilon:=2\pi d/\lambda$ we obtain
\begin{equation}\label{maxwell_dimensionless_1}
\bigg\{\frac{\eta}{\eta_0}(y)\bigg\}^{-1}\curl_y \widetilde{H}={\rm i}\frac{2\pi d}{\lambda} \widetilde{E}.
\end{equation}
Similarly, we carry out the dimensional analysis for the second equation in \eqref{maxwell_timeharm}, which yields  
\begin{equation}\label{maxwell_dimensionless_2}
\bigg\{\frac{\nu}{\nu_0}(y)\bigg\}^{-1}\curl_y \widetilde{E}= -{\rm i}\frac{\lambda}{2\pi d} \widetilde{H}.
\end{equation}

The non-dimensional form of the Maxwell system equations now follows ({\it cf.} (\ref{nondim}), for which $z=1$):
\begin{equation}
\label{maxwell_timeharmoperator}
\begin{pmatrix}
0 & \widetilde{A}\curl \\[0.3em]
-A\curl & 0
\end{pmatrix}
\begin{pmatrix}
H \\[0.3em]
E
\end{pmatrix}
+{\rm i}z
\begin{pmatrix}
H \\[0.3em]
E
\end{pmatrix}
=
\begin{pmatrix}
\widetilde{A} f \\[0.3em]
-Ag 
\end{pmatrix},
\end{equation}
where for brevity we have removed the tilde from the dimensionless fields and displacements.
Here $z=\lambda/(2\pi d)$ 
is the dimensionless parameter that appearing in \eqref{maxwell_dimensionless_1}, \eqref{maxwell_dimensionless_2}, $A=\eta/\eta_0$ is the inverse of the relative permittivity and $\widetilde{A}:=\nu/\nu_0$ is the inverse of the relative permeability. Furthermore, $g$ is a divergence-free function representing the external currents of the system, and $f$ is a divergence-free auxiliary function. 

In the spirit of the works \cite{Birman_Solomyak_89}. \cite{BS04}, one can consider the Maxwell operator $\mathfrak{M}$ given by the differential expression 
\[
\begin{pmatrix} 0 &  \widetilde{A}\curl\\[0.3em]
 -A\curl& 0 \end{pmatrix}
\]
acting on
\begin{align*}
{\rm dom}(\mathfrak{M})=\bigl\{(H, E)\in L^2(\R^3)&\oplus L^2(\R^3):\\[0.3em]
&\divv H=0, \,\divv E=0, \,A\curl H \in L^2(\R^3), \,\widetilde{A}\curl E \in L^2(\R^3)\bigr\},
\end{align*}
where $L^2(\R^3)$ is the space of $\C^3$-valued functions on ${\mathbb R}^3$ that are square-integrable with respect to the Lebesgue measure. 


Convergence estimates similar to those stated in Theorems \ref{main_thm}, \ref{corollary_main_thm},  \ref{thm_ED}, and \ref{cor_DE} are proved under the assumption that $-{\rm i}z\in K\cap\rho(\mathfrak{M})$, where $K\subset \R^3$ is compact and $\rho(\mathfrak{M})$ is the resolvent set of $\mathfrak{M}.$ In particular, $z=-{\rm i}$ corresponds to $\lambda=1$ in the formulation (\ref{lambda_res}), which is obtained from (\ref{maxwell_timeharmoperator}) by setting assuming that the medium is non-magnetic, i.e., $\nu=\nu_0$ and setting the second component of the right-hand side to zero.

\renewcommand{\theequation}{B.\arabic{equation}}
\renewcommand{\thesubsection}{B.\arabic{subsection}}
\setcounter{equation}{0}


\section*{Appendix B: Validity of (\ref{curlpoincare}) for some singular measures}
\label{measures_section}

Here we show that Assumption \ref{ass1} holds for the measures from the class (a) described at the end of Section \ref{Helmholtz_section}, and hence for the classes (b), (c). The validity of Assumption \ref{ass1} for the Lebesgue measure (example (d)) is shown easily via an argument based on the Fourier series, see {\it e.g.} \cite{CE16}.

Consider a finite set $\{{\mathcal P}_j\}_{j=1}^N$ of (two-dimensional) planes in ${\mathbb R}^3,$ such that each plane is orthogonal to one of the coordinate axes. Define the measure $\mu$ on $Q$ by the formula 
\begin{equation}
\mu(B)=N^{-1}\sum_{j=1}^N\vert {\mathcal P}_j\cap B\vert_j\ \ {\rm for\ all\ Borel\ } B\subset Q,
\label{mu}
\end{equation}
where $\vert\cdot\vert_j$ represents the $2$-dimensional Lebesgue measure. In what follows (see Section \ref{connect}), we will use the assumption that $(\cup_{j=1}^N{\mathcal P}_j)\cap Q$ is non-empty and connected.
For each $j=1,\dots, N,$ we also consider the measure $\mu_j$ defined by
\[
\mu_j(B):=
\vert {\mathcal P}_j\cap B\vert_j\ \ {\rm for\ all\ Borel\ } B\subset Q,
\]
so that $\mu=N^{-1}\sum_{j=1}^N\mu_j,$ see (\ref{mu}).

\subsection{Curls of zero for a measure supported by a plane}

\label{curl0}

In this section we fix $j\in\{1,\dots N\}.$ In line with Definition \ref{H1kappa_def},
we say that $v\in L^2(Q, d\mu_j)$ is a $\kappa$-curl of zero with respect to the measure $\mu_j$ if there exists a sequence $\{\phi_n\}\subset[C^\infty_\#\bigr]^3$ such that
\begin{equation*}
\int_Q |\phi_n|^2 d\mu_j \stackrel{n\to\infty}{\longrightarrow}0 \quad \quad \int_Q\bigl|\curl(e_\kappa\phi_n)-v\bigr|^2 d\mu_j \stackrel{n\to\infty}{\longrightarrow}0.
\end{equation*}

Without loss of generality, we can assume in what follows  that the plane ${\mathcal P}_j$ passes through zero
and is orthogonal to the $x_3$ direction.
\begin{prop}
\label{zero_prop}
For each $\kappa\in Q',$ the set of $\kappa$-curls of zero with respect to the measure $\mu_j$ coincides with 
\[
L^2_{\rm s}(Q, d\mu_j)\oplus L^2_{\rm s}(Q, d\mu_j)\oplus\{0\},
\] 
where (see Section \ref{Helmholtz_section}) $L^2_{\rm s}(Q, d\mu_j)$ is the space of ${\mathbb C}$-valued functions on $Q$ that are square integrable with respect to the measure $\mu_j.$
\end{prop}
\begin{proof}
For given functions $\xi_1, \xi_2\in L^2_{\rm s}(Q, d\mu_j),$ consider sequences of smooth $Q$-periodic functions, independent of $x_3,$
\[
\bigl\{\xi_j^{(n)}=\xi_j^{(n)}(x_1, x_2),\ \ n\in{\mathbb N}\bigr\},\qquad
j=1,2,
\]
such that
\[
\xi_j^{(n)}\stackrel{n\to\infty}{\longrightarrow}\xi_j\quad{\rm in}\ \  L^2_{\rm s}(Q, d\mu_j),\qquad j=1,2.
\]
Suppose also that functions
$\alpha=\alpha(x_3),$  $\beta=\beta(x_3)$ of the single variable $x_3$ are infinitely smooth and 1-periodic, and that
 their Taylor expansions at zero have the form $x_3+O(x_3^2).$
Define  
\begin{equation}
\phi_n(x)=\left(\begin{array}{c}\beta(x_3)\xi_2^{(n)}(x_1, x_2)\\[0.5em]-\alpha(x_3)\xi_1^{(n)}(x_1, x_2)\\[0.5em]0
\end{array}\right),\quad x=(x_1,x_2, x_3)\in Q,\qquad n\in{\mathbb N}.
\label{phi_zero}
\end{equation}
Then $\{\phi_n\}\subset[C^\infty_\#\bigr]^3$ and by a direct calculation one has, for all $n\in{\mathbb N},$
\[
\overline{e}_\kappa\curl(e_\kappa\phi_n)(x_1, x_2, x_3)=\left(\begin{array}{c}
\bigl(\alpha'(x_3)-{\rm i}\kappa_3\alpha(x_3)\bigr)\xi_1^{(n)}(x_1, x_2)\\[0.5em]\bigl(\beta'(x_3)+{\rm i}\kappa_3\beta(x_3)\bigr)\xi_2^{(n)}(x_1, x_2)
\\[0.5em]-\alpha(x_3)(\partial_1+{\rm i}\kappa_1)\xi^{(n)}_1(x_1, x_2)-\beta(x_3)(\partial_2+{\rm i}\kappa_2)\xi^{(n)}_2(x_1, x_2)\end{array}\right),
\]
where $\partial_j$ is the operator of differentiation with respect to the variable $x_j,$  $j=1,2.$
Due of the assumptions on $\alpha,$ $\beta,$ one has
\begin{equation*}
\int_Q |\phi_n|^2 d\mu_j=0\qquad \forall n,
\end{equation*}
and 
\begin{equation*}
\overline{e}_\kappa\curl(e_\kappa\phi_n)\stackrel{n\to\infty}{\longrightarrow}(\xi_1, \xi_2, 0)^\top
\quad {\rm in}\ \ L^2(Q, d\mu_j).
\end{equation*}
It follows that  $L^2_{\rm s}(Q, d\mu_j)\oplus L^2_{\rm s}(Q, d\mu_j)\oplus\{0\}$ is contained in the set of curls of zero. 

On the other hand, any vector of the form 
\[
(0, 0, \xi_3)^\top,
\qquad \xi_3\in L^2_{\rm s}(Q, d\mu_j),
\]
is orthogonal to all $\kappa$-cirls of zero. Indeed, for any sequence $\xi_3^{(n)}=\xi_3^{(n)}(x_1, x_2)$ of infinitely smooth $x_3$-independent functions converging to $\xi_3$ in $L^2_{\rm s}(Q, d\mu_j)$ and any sequence of vector functions
\[
\phi^{(n)}=\bigl(\phi_1^{(n)}, \phi_2^{(n)}, \phi_3^{(n)}\bigr)^\top
\in\bigl[C^\infty_\#\bigr]^3,\qquad n\in{\mathbb N},
\] 
such that
\[
\int_Q\bigl|\phi^{(n)}\bigr|^2 d\mu_j\stackrel{n\to\infty}{\longrightarrow}0, 
\]
one has (due to the fact that the integration by parts is carried out with respect to the variables $x_1,$ $x_2$ in the plane ${\mathcal P}_j$)
\begin{equation}
\int_Q\bigl((\partial_2+{\rm i}\kappa_2)\phi^{(n)}_1-(\partial_1+{\rm i}\kappa_1)\phi^{(n)}_2\bigr)\overline{\xi^{(n)}_3}d\mu_j=\int_Q\bigl(\phi_2\overline{(\partial_1+{\rm i}\kappa_1)\xi^{(n)}_3}-\phi_{1}\overline{(\partial_2+{\rm i}\kappa_2)\xi^{(n)}_3}\bigr)d\mu_j=0.
\label{xi_orth}
\end{equation}
It follows from (\ref{xi_orth}) that if $\curl(e_\kappa\phi^{(n)})\to v$ as $n\to\infty,$ then $\xi_3$ is orthogonal to $v_3$ in $L^2_{\rm s}(Q, d\mu_j),$ and therefore 
$(0, 0,\xi_3)^\top$ is orthogonal to $v$ in $L^2(Q, d\mu_j).$ Therefore, the set of $\kappa$-curls of zero is contained in $L^2_{\rm s}(Q, d\mu_j)\oplus L^2_{\rm s}(Q, d\mu_j)\oplus\{0\}.$ This concludes the proof of the claim that these two sets coincide.
\end{proof}

\subsection{Approximation in $H^1_{\curl,\kappa}(Q, d\mu)$ by smooth functions}

\label{smooth_approx}

He we prove the following auxiliary statement, which will allow us to establish (\ref{poincare_v})
by first showing that it holds for infinitely smooth functions.

\begin{lem}
\label{aux_lem}
Suppose that $(u, v)\in H^1_{\curl,\kappa}(Q, d\mu),$ where the function $u$ is solenoidal (see Section \ref{Sobolev_section})
\[
\overline{e}_\kappa{\rm div}(e_\kappa u)=0,
\]
and $\curl(e_\kappa u)$ is pointwise orthogonal to the support of measure $\mu.$
Then there exists a sequence $\{\phi_n\}\subset[C^\infty_\#]^3$ 
such that
\begin{equation}
\bigl(e_\kappa \phi_n, \curl(e_\kappa\phi_n)\bigr)\stackrel{n\to\infty}{\longrightarrow}(u, v)\ \ {\rm in}\ \ L^2(Q, d\mu) \oplus L^2(Q, d\mu)
\label{phin_approx}
\end{equation}
and the following two properties hold: 
\begin{equation}
 \overline{e}_\kappa{\rm div}(e_\kappa\phi_n)=0
\label{div_weak}
\end{equation}
in the sense of (\ref{weak_div_cond}) with $F=\phi_n,$ and the vector $\curl(e_\kappa\phi_n)$ is pointwise orthogonal to ${\rm supp}(\mu)$ (excluding the lines of intersection of the planes ${\mathcal P}_j,$ $j=1,\dots, N.$)

\end{lem}

\begin{proof}
According to Definition \ref{H1kappa_def}, there exists a sequence $\{\widetilde{\phi}_n, n\in{\mathbb N}\}\subset[C^\infty_\#]^3$ approximating $(u, v)$ in the sense that (\ref{phin_approx}) holds with $\phi_n$ replaced by $\widetilde{\phi}_n,$ however 
one does not necessarily have $\overline{e}_\kappa{\rm div}(e_\kappa\widetilde{\phi}_n)=0.$ 
In order to ``correct" the sequence $\{\widetilde{\phi}_n\},$ for each $n$ consider the solution $w_n\in H^1_\#$  (see {\it e.g.} Section \ref{Helmholtz_section}) to the elliptic problem
\begin{equation}
-\Delta w_n={\rm div}\bigl(e_\kappa\widetilde{\phi}_n\bigr)
\label{w_eq}
\end{equation}
understood in the weak sense with respect to the measure $\mu:$
\begin{equation}
\int_Q\nabla w_n\cdot \nabla\varphi d\mu=-\int_Qe_\kappa\widetilde{\phi}_n\cdot \nabla\varphi d\mu\qquad \forall \varphi\in C^\infty_\#.
\label{w_varform}
\end{equation}
The problem (\ref{w_varform}) has a unique solution $(w_n, \nabla w_n)\in H^1_\#,$ which has the property that $\nabla w_n$ is orthogonal to all {\it gradients of zero} \cite{Zhikov2000} with respect to the measure $\mu:$ indeed, setting $\varphi=\varphi_j,$ $j\in{\mathbb N},$ in (\ref{w_varform}), where 
\[
\int_Q|\varphi_j|^2d\mu\stackrel{j\to\infty}{\longrightarrow}0,\qquad \int_Q|\nabla\varphi_j-v|^2d\mu\stackrel{j\to\infty}{\longrightarrow}0\
\qquad v\in L^2(Q, d\mu),
\]
and passing in the obtained identity to the limit as $j\to\infty$ yields
\[
\int_Q\nabla w_n\cdot v\,d\mu=0,
\]
as claimed. Following an argument similar to that given in \cite[Section 3.1]{Zhikov2000}, see also \cite[Section 4]{Zhikov2002}, it is shown that the set of gradients of zero is a closed subspace of $L^2(Q, d\mu)$ consisting of vector functions that, when restricted to the plane ${\mathcal P}_j,$ $j=1,\dots, N,$ are pointwise orthogonal to it. Furthermore, it is straightforward to show that for each $n$ the function $w_n$ is infinitely smooth on $Q\cap{\mathcal P}_j,$ {\it e.g.} by deducing the decay properties of the coefficients of its Fourier series with respect to $x_1, x_2$ in terms of the decay properties, as $n\to\infty,$ of the Fourier coefficients of $\widetilde{\phi}_n.$
 
For each $n\in{\mathbb N},$ we consider an infinitely smooth function $\widetilde{w}_n$ on $Q$ that for each $j\in\{1,\dots, N\}$ coincides with $w_n$ on $Q\cap{\mathcal P}_j$  and has zero gradient in the variables orthogonal to ${\mathcal P}_j.$ (Such a smooth extension from $(\cup_{j=1}^N{\mathcal P}_j)\cap Q$ to $Q$ can be obtained in a standard way  by an appropriate partition of unity on $Q,$ carrying out standard extensions in corner, edge, and face regions, and using appropriate mollifiers.) 
Clearly, on ${\rm supp}(\mu)$ one has 
\begin{equation}
{\rm curl}\bigl(e_\kappa(\overline{e}_\kappa\nabla\widetilde{w}_n)\bigr)={\rm curl}\,(\nabla\widetilde {w}_n)={\rm curl}\,(\nabla{w}_n)=0.
\label{curl_zero}
\end{equation}
Furthermore, writing (\ref{w_varform}) in the form (where we take advantage of $u$ being solenoidal)
\[
\int_Q\nabla w_n\cdot\nabla\varphi d\mu=\int_Qe_\kappa(u-\widetilde{\phi}_n)\cdot\nabla\varphi d\mu\qquad \forall \varphi\in C^\infty_\#,
\]
setting $\varphi=w_n,$ and using the fact that the right-hand side of the result goes to zero as $n\to0,$ we obtain 
\[
\int_Q\bigl|\nabla\widetilde{w}_n\bigr|^2d\mu=\int_Q|\nabla w_n|^2d\mu\stackrel{n\to\infty}{\longrightarrow}0.
\]
Combining this observation with (\ref{curl_zero}) and (\ref{w_eq}), we conclude that the functions 
\[
\widehat{\phi}_n:=\widetilde{\phi}_n+\overline{e}_\kappa\nabla\widetilde{w}_n,\qquad n\in{\mathbb N},
\] 
are smooth and have the convergence properties 
\[
\int_Q\bigl|\widehat{\phi}_n-u\bigr|^2d\mu\stackrel{n\to\infty}{\longrightarrow}0,\qquad \int_Q\bigl|{\rm curl}\bigl(e_\kappa \widehat{\phi}_n\bigr)-v\bigr|^2d\mu\stackrel{n\to\infty}{\longrightarrow}0,
\]
and $e_\kappa\widehat{\phi}_n$ is solenoidal for each $n\in{\mathbb N},$ as required in (\ref{phin_approx}), (\ref{div_weak}).

In order to fulfil the second property claimed in the lemma, we construct a further ``correction" to the sequence $\{\widetilde{\phi}_n\},$ which does not affect the properties (\ref{phin_approx}), (\ref{div_weak}). For each $j\in\{1,\dots N\},$ 
consider the rotation $R_j$ in ${\mathbb R}^3$ such that the plane $R_j{\mathcal P}_j$ passes through zero
and is orthogonal to the $x_3$ direction. To simplify the notation, we fix $j$ and assume, as in Section \ref{curl0}, that $R_j=I.$

Under the above convention, notice that the projection of $\curl(e_\kappa\widehat{\phi}_n)$ onto the plane ${\mathcal P}_j,$ restricted to the set 
${\mathcal P}_j\cap Q$ ({\it i.e.} the support of $\mu_j$)  is a smooth function 
\[
\psi_n=
e_\kappa\bigl((\psi_n)_1(x_1, x_2), (\psi_n)_2(x_1, x_2), 0\bigr)^\top,
\quad (x_1, x_2)\in[0, 1)^2.
\]
where, using the notation $\widehat{\phi}_n=((\widehat{\phi}_n)_1, (\widehat{\phi}_n)_2, (\widehat{\phi}_n)_3),$
\begin{equation*}
\begin{aligned}
(\psi_n)_1=
\bigl(({\rm i}\kappa_2+\partial_2)(\widehat{\phi}_n)_3-({\rm i}\kappa_3+\partial_3)
(\widehat{\phi}_n)_2\bigr)\bigr\vert_{x_3=0},\\[0.4em]
 (\psi_n)_2=
 \bigl(({\rm i}\kappa_3+\partial_3)(\widehat{\phi}_n)_1-({\rm i}\kappa_1+\partial_1)
 (\widehat{\phi}_n)_3\bigr)\bigr\vert_{x_3=0}.
\end{aligned}
\end{equation*}
Consider the vector ({\it cf.} (\ref{phi_zero}))
\begin{equation*}
\widehat{\psi}_n(x)=\left(\begin{array}{c}\beta(x_3)(\psi_n)_1(x_1, x_2)\\[0.5em]-\alpha(x_3)(\psi_n)_2(x_1, x_2)\\[0.5em]0
\end{array}\right),\quad x=(x_1, x_2, x_3)\in Q,\qquad n\in{\mathbb N},
\end{equation*}
where the functions
$\alpha=\alpha(x_3),$  $\beta=\beta(x_3)$ of the single variable $x_3$ are infinitely smooth and 1-periodic, and that
 their Taylor expansions at zero have the form $x_3+O(x_3^2).$
Similarly to the argument in the proof of Proposition \ref{curl_zero}, we notice that 
$\curl(e_\kappa\widehat{\psi}_n)=\psi_n,
$
now viewed as a function on the whole of $Q.$ Furthermore, the vector $\widehat{\psi}_n$ is trivially solenoidal, as the vector $\widehat{\psi}_n$ vanishes on ${\mathcal P}_j\cap Q,$ and $\curl(e_\kappa\widehat{\psi}_n)\to 0$ in $L^2(Q, d\mu)$ as $n\to\infty,$ due to the assumption of pointwise orthogonality of $\curl(e_\kappa{\mathfrak u})$ to ${\mathcal P}_j\cap Q.$

The above construction is repeated for each $j\in\{1,\dots, N\},$ now taking into account that it will be preceded by the rotation $R_j.$ Relabel by $\widehat{\psi}_n^{(j)}$ the elements of the constructed sequence $\widehat{\psi}$. As a result, the sequence 
\[
\phi_n=\widehat{\phi}_n-\sum_{j=1}^NR_j^\top\widehat{\psi}_n^{(j)},\qquad n\in{\mathbb N}
\] 
satisfies all the required conditions.
\end{proof}

\subsection{Poincar\'{e} inequality}

In this section we carry out the proof of the Poincar\'{e}-type inequality ({\it cf.}  (\ref{poincare_v}))
\begin{equation}
\biggl\|{\mathfrak u}-\int_Q {\mathfrak u}\biggr\|_{L^2(Q,d\mu_j)}\leq C_{\rm P}\bigl\|\curl(e_\kappa {\mathfrak u})\bigr\|_{L^2(Q,d\mu_j)}
\label{uj}
\end{equation}
for functions $\mathfrak u$ satisfying $\overline{e}_\kappa{\rm div}\,(e_\kappa{\mathfrak u})=0,$ in the case of the measure $\mu$ defined by (\ref{mu}). 

Notice that in the inequality (\ref{uj}) 
we can assume, without loss of generality, that the vector $\curl(e_\kappa{\mathfrak u})$ is orthogonal to ${\mathcal P}_j$ at almost every point of ${\mathcal P}_j\cap Q.$ Indeed, one can write 
\[
\curl(e_\kappa {\mathfrak u})=w_1+w_2,
\] 
where $w_2$ is the projection of $\curl(e_\kappa{\mathfrak u})$ onto the subspace of $L^2(Q,d\mu)$ consisting of $\kappa$-curls of zero,  $w_1$ is another value of the $\kappa$-curl of ${\mathfrak u},$ so that $w_1$ and $w_2$ are orthogonal in the sense of $L^2(Q,d\mu).$ As we showed in Section \ref{curl0}, in the case of the measure $\mu_j$ $\kappa$-curls of zero are parallel to ${\mathcal P}_j$ at each point, so $w_1$ is pointwise parallel to ${\mathcal P}_j$ and $w_2$ is pointwise orthogonal to ${\mathcal P}_j.$ In what follows we can therefore assume
that $\curl(e_\kappa {\mathfrak u})$ is orthogonal to ${\mathcal P}_j.$ This will allow us, in particular, to use Lemma \ref{aux_lem}.

We first prove an auxiliary proposition reflecting the vectorial nature of the inequality (\ref{uj}), due to the presence of the operator curl and then combine it with the ``scalar" Poincar\'{e} inequality applied to each component of the vector ${\mathfrak u}.$  

Having proved (\ref{poincare_v}) with the measure $\mu$ replaced by $\mu_j,$ we will then show, in Section \ref{connect} that it holds for $\mu$ as well (possibly with a larger constant $C_{\rm P}$), using the assumption that the the set 
\[
(\cup_{j=1}^N{\mathcal P}_j)\cap Q=\cup_{j=1}^N({\mathcal P}_j\cap Q)
\] 
is connected.

\subsubsection{The norm of the transversal curl is the norm of the tangential gradient}

\label{curl_bound}

Here we fix $j\in\{1,\dots, N\}$ and, as in Section \ref{curl0}, assume that the plane 
${\mathcal P}_j$ passes through zero
and is orthogonal to the $x_3$ direction. For $\kappa\in Q'$ and a function 
$\phi\in[C_\#^\infty]^3,$ we denote by  $\widetilde{\nabla}(e_\kappa\phi)$ the pointwise orthogonal projection of the $\nabla(e_\kappa\phi)$ onto the $(x_1, x_2)$-plane. 



\begin{prop}
\label{curl_grad_prop}
Suppose that a vector function $\phi\in[C^\infty_\#]^3$ is solenoidal, i.e. ({\it cf.} (\ref{div_weak}))
\begin{equation}
\overline{e}_\kappa{\rm div}(e_\kappa\phi)=0
\label{sol_phi}
\end{equation}
and that 
 the vector $\curl(e_\kappa\phi)$ is pointwise parallel to $x_3$ at each point of ${\mathcal P}_j\cap Q.$ 
Then, for all $\kappa\in Q',$ one has 
\begin{equation*}
\bigl\|\widetilde{\nabla}
(e_\kappa \phi)\bigr\|_{L^2(Q,d\mu_j)}^2
=\bigl\|\curl(e_\kappa \phi)\bigr\|^2_{L^2(Q,d\mu_j)},
\end{equation*}
\end{prop}
\begin{proof}

We 
we expand the function $\phi$
into the standard Fourier series:
\begin{equation}
\phi(x)=\sum_{l\in{\mathbb Z}^3}\exp(2\pi{\rm i} l\cdot x){c}_l, \quad x\in Q,\qquad c_l\in{\mathbb C}^3,\ l\in{\mathbb Z}^3
\label{phi_FS}
\end{equation}
and notice that 
\[
\curl(e_\kappa\phi_n)(x)={\rm i}e_\kappa\sum_{l\in{\mathbb Z}^3}\exp({\rm i}2\pi l\cdot x)\bigl(c_l\times (\kappa+2\pi l)\bigr), \quad x\in Q,
\] 
where the series converges in the norm of $L^2(Q),$ with respect to the Lebesgue measure on $Q.$ Since 
$\curl(e_\kappa{\mathfrak\phi})$ is pointwise orthogonal to ${\mathcal P}_j,$ it follows that 
for each $l\in{\mathbb Z}^3$
the vector $c_l\times(\kappa+2\pi l)$ is orthogonal to ${\mathcal P}_j,$ {\it i.e.} it is parallel to the $x_3$ direction. 




 
For each $x\in Q,$ we denote $(x_1, x_2)=:\widetilde{x}.$ 
and, similarly, for each  value $\kappa\in Q'$ of the quasimomentum, we denote $\widetilde{\kappa}:=(\kappa_1, \kappa_2).$ 
Finally, for each ``multi-index" $l\in{\mathbb Z}^3,$ we consider
 the ``sub-index" $\widetilde{l}:=(l_1, l_2)\in{\mathbb Z}^2.$
We write finite truncations of (\ref{phi_FS}) in the form ($K\in{\mathbb N}$)
\begin{equation}
\phi_K(x)=\sum_{|l|\le K}
 {c}_l  \exp(2\pi{\rm i} l\cdot x)=\sum_{|\widetilde{l}|\le K}
\ \ \sum_{ |{l_3}|\le K-|\widetilde{l}|}
 {c}  _{\widetilde{l}, {l_3}}\exp\bigl(2\pi{\rm i}(\widetilde{l}, {l_3})\cdot (\widetilde{x},{x_3})\bigr), \quad x\in Q.
\label{phi_sub1}
\end{equation}
In the remainder of this section, for brevity, we omit the 
summation ranges for $\widetilde{l},$ $l_3,$ which are the same as in (\ref{phi_sub1}) throughout. From (\ref{phi_sub1}) one has   
\begin{equation}
\begin{aligned}
\int_Q\bigl\vert{\rm i}\varphi\kappa+\nabla\phi_K\bigr\vert^2d\mu_j=\sum_{\widetilde{l}}&\biggl(\sum_{{l_3}}
\Bigl\{ {c}_{({\widetilde{l}, {l_3}})}\otimes\bigl(\widetilde{\kappa}+2\pi\widetilde{l}, {\kappa_3}+2\pi{l_3}\bigr)\Bigr\}\biggr)^\top\\[0.4em]
&\times\biggl(\sum_{{m_3}}
\Bigl\{\overline{ {c}}_{({\widetilde{l}, {m_3}})}\otimes\bigl(\widetilde{\kappa}+2\pi\widetilde{l}, {\kappa_3}+2\pi{m_3}\bigr)\Bigr\}\biggr).
\end{aligned}
\label{product}
\end{equation}
Rearranging the product under the external summation in (\ref{product}) yields\footnote{Recall that by $a\cdot b$ we denote the sesquilinear inner product of 
$a, b\in{\mathbb C}^3.$}
\begin{equation*}
\begin{aligned}
&\int_Q\bigl\vert{\rm i}\varphi\kappa+\nabla\phi_K\bigr\vert^2d\mu_j\\[0.5em]
&=\sum_{\widetilde{l}}\sum_{{l_3}, {m_3}}\bigl\{ {c}_{(\widetilde{l}, {l_3})}\cdot
{ {c}_{(\widetilde{l}, {m_3})}}\bigr\}\bigl\{(\widetilde{\kappa}+2\pi\widetilde{l}, {\kappa_3}+2\pi{l_3})\cdot(\widetilde{\kappa}+2\pi \widetilde{l}, {\kappa_3}+2\pi{m_3})\bigr\}.
\end{aligned}
\end{equation*}
 
 In order to manipulate the above expression into a convenient form, we notice two properties of the Fourier series for $\phi,$ due to the assumptions that it is solenoidal (see (\ref{sol_phi}))
 and that its $\kappa$-curl is orthogonal to ${\mathcal P}_j$ ({\it i.e.} parallel to the $x_3$-direction).
 In terms of the Fourier coefficients $c_l,$ the first condition can be written as follows:
\begin{equation}
\begin{aligned}
0&=\int_{Q}\sum_{\widetilde{p}, l_3}
\exp\bigl({\rm i}(\widetilde{p}\cdot\widetilde{x}+l_3x_3)\bigr)
c_{(\widetilde{p}, l_3)} \cdot 
\exp\bigl(-{\rm i}(\widetilde{l}\cdot\widetilde{x}+m_3x_3)\bigr)
(\widetilde{\kappa}+2\pi\widetilde{l}, {\kappa_3}+2\pi{m_3})d\mu_j\\[0.4em]
&=\int_{(0,1)^2}\sum_{\widetilde{p}, l_3}
\exp({\rm i}\widetilde{p}\cdot\widetilde{x})
c_{(\widetilde{p}, l_3)} \cdot 
\exp(-{\rm i}\widetilde{l}\cdot\widetilde{x})
(\widetilde{\kappa}+2\pi\widetilde{l}, {\kappa_3}+2\pi{m_3})d\widetilde{x}\\[0.4em]
&=\sum_{{l_3}}
c_{(\widetilde{l}, {l_3})} \cdot(\widetilde{\kappa}+2\pi\widetilde{l}, {\kappa_3}+2\pi{m_3})
\qquad \forall\,\widetilde{l}\in{\mathbb Z}^2,\,m_3\in{\mathbb Z},
\end{aligned}
\label{div_cond}
\end{equation}
which is obtained by setting 
\[
\varphi(x)=
\exp\bigl(-{\rm i}\widetilde{l}\cdot\widetilde{x}+m_3x_3\bigr),\quad x\in Q,\qquad \widetilde{l}\in{\mathbb Z}^2,\,m_3\in{\mathbb Z},
\]
as the test function for (\ref{sol_phi}).

Similarly, the second condition takes the form
\begin{equation}
\begin{aligned}
0&=\int_Q\sum_{\widetilde{p}, l_3}
\exp({\rm i}\widetilde{p}\cdot\widetilde{x})
\bigl(c_{(\widetilde{p}, l_3)} \times(\widetilde{\kappa}+2\pi\widetilde{p}, {\kappa_3}+2\pi{l_3})\bigr)\cdot \exp(-{\rm i}\widetilde{l}\cdot x)a,\\[0.4em]
&=\sum_{{l_3}}
\Bigl(c_{(\widetilde{l}, {l_3})} \times\bigl(\widetilde{\kappa}+2\pi\widetilde{l}, {\kappa_3}+2\pi{l_3}\bigr)\Bigr)\cdot a\qquad \forall\,\widetilde{l}\in{\mathbb Z}^2,\,a\in(0,1)^2.
\end{aligned}
\label{curl_cond}
\end{equation}

Using standard formulae of vector calculus, we write, for each $\widetilde{l}\in{\mathbb Z}^2,$ $m_3\in{\mathbb Z},$
\begin{align*}
&\sum_{{l_3}}
\Bigl( {c}_{(\widetilde{l}, {l_3})}\times(\widetilde{\kappa}+2\pi\widetilde{l}, {\kappa_3}+2\pi{l_3})\Bigr)\cdot\,
\Bigl({ {c}_{(\widetilde{l}, {m_3})}}\times
 (\widetilde{\kappa}+2\pi\widetilde{l}, {\kappa_3}+2\pi{m_3})\Bigr)\\[0.5em]
 &=\sum_{{l_3}}
 {c}_{(\widetilde{l}, {m_3})}\cdot\Bigl\{(\widetilde{\kappa}+2\pi\widetilde{l}, {\kappa_3}+2\pi{m_3})\times
 \Bigl({ {c}_{(\widetilde{l}, {l_3})}}\times
 (\widetilde{\kappa}+2\pi\widetilde{l}, {\kappa_3}+2\pi{l_3})\Bigr)\Bigr\}\\[0.5em]
 &=\sum_{{l_3}}
 {c}_{(\widetilde{l}, {m_3})}\cdot\Bigl\{(\widetilde{\kappa}+2\pi\widetilde{l}, 0)\times
 \Bigl({ {c}_{(\widetilde{l}, {l_3})}}\times
 (\widetilde{\kappa}+2\pi\widetilde{l}, {\kappa_3}+2\pi{l_3})\Bigr)\Bigr\}\\[0.5em]
 &
 \qquad\qquad\qquad+{c}_{(\widetilde{l}, {m_3})}\cdot
 \Bigl\{(0, {\kappa_3}+2\pi{m_3})\times\sum_{{l_3}}
\Bigl({ {c}_{(\widetilde{l}, {l_3})}}\times
 (\widetilde{\kappa}+2\pi\widetilde{l}, {\kappa_3}+2\pi{l_3})\Bigr)\Bigr\}\\[0.5em]
 &=\sum_{{l_3}}
 {c}_{(\widetilde{l}, {m_3})}\cdot\Bigl\{(\widetilde{\kappa}+2\pi\widetilde{l}, 0)\times\Bigl(
 { {c}_{(\widetilde{l}, {l_3})}}\times
 (\widetilde{\kappa}+2\pi\widetilde{l}, {\kappa_3}+2\pi{l_3})\Bigr)\Bigr\}\\[0.5em]
 &=\sum_{{l_3}}
 \Bigl(\bigl\{ {c}_{(\widetilde{l}, {l_3})}\cdot
 { {c}_{(\widetilde{l}, {m_3})}}\bigr\}\bigl\{(\widetilde{\kappa}+2\pi\widetilde{l}, 0)\cdot(\widetilde{\kappa}+2\pi \widetilde{l}, 0)\bigr\}\\[0.5em]
&
\qquad\qquad\qquad-\bigl\{ {c}_{(\widetilde{l}, {l_3})}\cdot (\widetilde{\kappa}+2\pi\widetilde{l}, {\kappa_3}+2\pi{m_3})\bigr\}\bigl\{
\overline{ {c}}_{(\widetilde{l}, {m_3})}\cdot (\widetilde{\kappa}+2\pi\widetilde{l}, 0)\bigr\}\Bigr)\\[0.5em]
 &=\sum_{{l_3}}
 \bigl\{ {c}_{(\widetilde{l}, {l_3})}\cdot
 { {c}_{(\widetilde{l}, {m_3})}}\bigr\}\bigl\{(\widetilde{\kappa}+2\pi\widetilde{l})\cdot(\widetilde{\kappa}+2\pi \widetilde{l})\bigr\}\\[0.5em]
&
\qquad\qquad\qquad-\Bigl\{\sum_{{l_3}}
{c}_{(\widetilde{l}, {l_3})}\cdot (\widetilde{\kappa}+2\pi\widetilde{l}, {\kappa_3}+2\pi{m_3})\Bigr\}\bigl\{\overline{ {c}}_{(\widetilde{l}, {m_3})}\cdot (\widetilde{\kappa}+2\pi\widetilde{l}, 0)\bigr\}\\[0.5em]
&=\sum_{{l_3}}
\bigl\{ {c}_{(\widetilde{l}, {l_3})}\cdot
{ {c}_{(\widetilde{l}, {m_3})}}\bigr\}\bigl\{(\widetilde{\kappa}+2\pi\widetilde{l})\cdot(\widetilde{\kappa}+2\pi \widetilde{l})\bigr\}.
\end{align*}
Here, for the third equality we use the fact that by (\ref{curl_cond}), 
the vector  
\[
\sum_{{l_3}}
\Bigl({ {c}_{(\widetilde{l}, {l_3})}}\times(\widetilde{\kappa}+2\pi\widetilde{l}, {\kappa_3}+2\pi{l_3})\Bigr)
\]
is orthogonal to the $(x_1, x_2)$-plane
and hence parallel to the vector $(0, {\kappa_3}+2\pi{m_3}),$ and for the sixth equality we use (\ref{div_cond}).
It follows that
\begin{align*}
&\int_Q\bigl\vert\nabla(e_\kappa\phi)\bigr\vert^2d\mu_j=\lim_{K\to\infty}\int_Q\bigl\vert{\rm i}\varphi\kappa+\nabla\phi_K\bigr\vert^2d\mu_j\\[0.6em]
&=\lim_{K\to\infty}\sum_{{|\widetilde{l}|\le K}}\ \ \sum_{\substack{|{l_3}|\le K-|\widetilde{l}|\\
|{m_3}|\le K-|\widetilde{l}|}}
\biggl\{\bigl\{ {c}_{(\widetilde{l}, {l_3})}  \times(\widetilde{\kappa}+2\pi\widetilde{l}, {\kappa_3}+2\pi{l_3})\bigr\}\cdot\,
\bigl\{{ {c}_{(\widetilde{l}, {m_3})}  }\times
 (\widetilde{\kappa}+2\pi\widetilde{l}, {\kappa_3}+2\pi{m_3})
 \bigr\}\\[0.6em]
&\qquad\qquad\qquad+\bigl\{c_{({\widetilde{l}, {l_3}})}  \cdot 
{c_{({\widetilde{l}, {m_3}})}  }\bigr\}({\kappa_3}+2\pi {l_3})\cdot({\kappa_3}+2\pi {m_3})\biggr\}
\\[0.5em]
&=\lim_{K\to\infty}\biggl\{\int_Q\bigl\vert({\rm i}\kappa+\nabla)\times\phi_K\bigr\vert^2d\mu_j+\int_Q\bigl\vert({\rm i}{\kappa_3}+{\partial_3})\phi_K\bigr\vert^2d\mu_j\biggr\}\\[0.6em]
&=\lim_{K\to\infty}\biggl\{\int_Q\bigl\vert\curl(e_\kappa\phi_K)\bigr\vert^2d\mu_j+\int_Q\bigl\vert{\partial_3}(e_\kappa\phi_K)\bigr\vert^2d\mu_j\biggr\}
=\int_Q\bigl\vert\curl(e_\kappa\phi)\bigr\vert^2d\mu_j+\int_Q\bigl\vert{\partial_3}(e_\kappa\phi)\bigr\vert^2d\mu_j,
\end{align*}
and therefore
\begin{align*}
\bigl\|\widetilde{\nabla}
(e_\kappa\phi)\bigr\|_{L^2(Q,d\mu_j)}^2=
\bigl\|\nabla(e_\kappa\phi)\bigr\|_{L^2(Q,d\mu_j)}^2-\bigl\|{\partial_3}(e_\kappa\phi)\bigr\|_{L^2(Q,d\mu_j)}^2
=\bigl\|\curl(e_\kappa\phi)\bigr\|^2_{L^2(Q,d\mu_j)},
\end{align*}
as required.
\end{proof}


\subsubsection{``Scalar" Poincar\'{e} inequality for a single plane}

We continue working with a fixed $j\in\{1,\dots, N\}$ and assume, without loss of generality,
that the plane 
${\mathcal P}_j$ passes through zero
and is orthogonal to the $x_3$-direction. For a function $\phi\in C_\#^\infty,$ 
we denote by $\widetilde{\nabla}\phi(x)\in{\mathbb R}^2,$ $x\in Q,$ the (pointwise)  projection of its gradient onto the $(x_1, x_2)$-plane.

We write 
\[
\phi(\widetilde{x})-\int_Q\phi d\mu_j=\sum_{\widetilde{l}\in{\mathbb Z}^2\setminus\{0\}}c_{\widetilde{l}}\exp\bigl(2\pi{\rm i}\widetilde{l}\cdot\widetilde{x}\bigr), \quad \widetilde{x}\in [0, 1)^2,\qquad c_{\widetilde{l}}\in{\mathbb C},\quad\widetilde{l}\in{\mathbb Z}^2\setminus\{0\},
\]
and notice that, for each $\widetilde{\kappa}\in[-\pi, \pi)^2,$ one has, assuming $\phi$ is non-constant on ${\mathcal P}_j\cap Q,$
\begin{align*}
&\biggl(\int_Q\biggl\vert\phi-\int_Q\phi d\mu_j\biggr\vert^2d\mu_j\biggr)^{-1}\int_Q\bigl\vert{\rm i}\varphi\widetilde{\kappa}+\widetilde{\nabla}\phi\bigr\vert^2d\mu_j\\[0.5em]
&=\biggl(\sum_{\widetilde{l}, \widetilde{m}\in{\mathbb Z}^2\setminus\{0\}}\alpha_{\widetilde{l}\widetilde{m}}c_{\widetilde{l}}\overline{c_{\widetilde{m}}}\biggr)^{-1}\biggl(\sum_{\widetilde{l}, \widetilde{m}\in{\mathbb Z}^2\setminus\{0\}}\alpha_{\widetilde{l}\widetilde{m}}c_{\widetilde{l}}\overline{c_{\widetilde{m}}}(\widetilde{\kappa}+2\pi\widetilde{l})\cdot(\widetilde{\kappa}+2\pi\widetilde{m})\biggr),
\end{align*}
where
\[
\alpha_{\widetilde{l}\widetilde{m}}:=\int_{(0,1)^2}\exp\bigl(2\pi{\rm i}(\widetilde{l}-\widetilde{m})\cdot \widetilde{x}\bigr)d\widetilde{x}=\left\{\begin{array}{ll}1, \ \ \widetilde{l}=\widetilde{m},\\[0.4em]
0\ \ \ {\rm otherwise}.\end{array}\right.
\]
It follows that 
\begin{equation*}
\biggl(\int_Q\biggl\vert\phi-\int_Q\phi d\mu_j\biggr\vert^2d\mu_j\biggr)^{-1}\int_Q\bigl\vert{\rm i}\varphi\widetilde{\kappa}+\widetilde{\nabla}\phi\bigr\vert^2d\mu_j
=\biggl(\sum_{\widetilde{l}\in{\mathbb Z}^2\setminus\{0\}}|c_l|^2\biggr)^{-1}\biggl(\sum_{\widetilde{l}, \widetilde{m}\in{\mathbb Z}^2\setminus\{0\}}
|c_l|^2|\widetilde{\kappa}+2\pi\widetilde{l}|^2\biggr)\ge\pi^2,
\end{equation*}
and hence 
\begin{equation}
\int_Q\biggl\vert\phi-\int_Q\phi d\mu_j\biggr\vert^2d\mu_j\le \pi^{-2}\bigl\|\widetilde{\nabla}
(e_\kappa\phi)\bigr\|_{L^2(Q,d\mu_j)}^2,
\label{lem_scal}
\end{equation}
where $\widetilde{\nabla}(e_\kappa\phi)$ is the ``tangential" gradient introduced at the beginning of Section \ref{curl_bound}. 
If the function $\phi$ is constant on ${\mathcal P}_j\cap Q,$ the inequality (\ref{lem_scal}) is satisfied trivially. Note also that an inequality of the same for as (\ref{lem_scal}) has thus been established for vector functions $\phi\in [C_\#]^3,$ by applying it component-wise and adding the inequalities obtained for the individual components. Below we discuss the vector case, for which the Poincar\'{e} inequality for any of the measures $\mu_j,$ $j=1,\dots, N,$ looks the same as (\ref{lem_scal}), where $\phi$ is now a smooth vector function.

\subsubsection{Connectivity argument}
\label{connect}

For the measure $\mu=\sum_{j=1}^N\mu_j$ and $\phi\in[C^\infty_\#]^3,$ we denote by 
$\widetilde{\nabla}(e_\kappa\phi)$ the (component-wise) tangential gradient of $\phi$ at points of 
${\rm supp}(\mu),$ {\it i.e.} the orthogonal projection of $\nabla(e_\kappa\phi)$ onto ${\rm supp}(\mu).$   

Suppose that for $j, k\in\{1,\dots, N\}$ the planes ${\mathcal P}_j$ and ${\mathcal P}_k$  intersect and fix a point $\alpha_{jk}\in{\mathcal P}_j\cap{\mathcal P}_k\cap Q.$ For any $\kappa\in Q',$ any function 
$\phi\in[C^\infty_\#]^3,$ and all $x\in{\mathcal P}_j\cap Q$, $y\in{\mathcal P}_k\cap Q,$ one has
\begin{equation}
\begin{aligned}
e_\kappa(x)\phi(x)-e_\kappa(y)\phi(y)&=\int_{0}^1\nabla(e_\kappa\phi)\bigl(\alpha_{jk}+t(x-\alpha_{jk})\bigr)dt\cdot(x-\alpha_{jk})\\[0.5em]
&-\int_{0}^1\nabla(e_\kappa\phi)\bigl(\alpha_{jk}+t(y-\alpha_{jk})\bigr)dt\cdot(y-\alpha_{jk}).
\end{aligned}
\label{both_sides}
\end{equation} 
Multiplying both sides of (\ref{both_sides}) by $e_\kappa(y)^{-1}=e_\kappa(-y)$ and integrating over $y\in Q$ with respect to the measure $\mu_k$ (recalling that ${\rm supp}(\mu_k)={\mathcal P}_k\cap Q$) yields
\begin{equation}
\begin{aligned}
e_\kappa(x)\phi(x)\int_Qe_\kappa^{-1}d\mu_k-\int_Q\phi d\mu_k\le
\sqrt{2}\biggl(\int_{0}^1&\bigl\vert\widetilde{\nabla}(e_\kappa\phi)\bigl(\alpha_{jk}+t(x-\alpha_{jk})\bigr)\bigr\vert dt\\[0.5em]
&+\bigl\Vert\widetilde{\nabla}(e_\kappa\phi)\bigr\Vert_{L^2(Q,d\mu_k)}\biggr)\qquad\forall x\in{\mathcal P}_j\cap Q.
\end{aligned}
\label{interm_phi}
\end{equation}
Furthermore, multiplying both sides of (\ref{interm_phi}) by $e_\kappa(x)^{-1}$ and integrating over $x\in Q$ with respect to the measure $\mu_j$ yields
\begin{equation*}
\int_Q\phi d\mu_j-\int_Q\phi d\mu_k\le 
\sqrt{2}\Bigl(\bigl\Vert\widetilde{\nabla}(e_\kappa\phi)\bigr\Vert_{L^2(Q,d\mu_j)}
+\bigl\Vert\widetilde{\nabla}(e_\kappa\phi)\bigr\Vert_{L^2(Q,d\mu_k)}\Bigr)\le\sqrt{2}\Vert\widetilde{\nabla}(e_\kappa\phi)\bigr\Vert_{L^2(Q,d\mu)}.
\end{equation*}
By interchanging $k$ and $j$ if necessary, we thus obtain 
\[
\biggl\vert\int_Q\phi d\mu_j-\int_Q\phi d\mu_k\biggr\vert\le \sqrt{2}\bigl\Vert\widetilde{\nabla}(e_\kappa\phi)\bigr\Vert_{L^2(Q,d\mu)}.
\]

Next, notice that since $(\cup_{j=1}^N{\mathcal P}_j)\cap Q$ is connected by assumption, for each pair of planes in the union there is a ``path" from one plane to the other involving at most $N$ planes, such that any ``adjacent" planes in the path intersect. It follows that for all pairs $j, k$ the following bound holds:
\begin{equation}
\biggl\vert\int_Q\phi d\mu_j-\int_Q\phi d\mu_k\biggr\vert\le \sqrt{2}N\bigl\Vert\widetilde{\nabla}(e_\kappa\phi)\bigr\Vert_{L^2(Q,d\mu)}.
\label{anyjk}
\end{equation}

Finally, using (\ref{anyjk}) and standard arithmetic inequalities, we obtain
\begin{align*}
\int_Q\biggl\vert\phi-\int_Q\phi\biggr\vert^2d\mu&=\sum_{j=1}^N\int_Q\biggl\vert\phi-\int_Q\phi\biggr\vert^2d\mu_j=\sum_{j=1}^N\int_Q\biggl\vert\phi-\sum_{k=1}^NN^{-1}
\int_Q\phi d\mu_k\biggr\vert^2d\mu_j\\[0.4em]
&=\sum_{j=1}^N\int_Q\biggl\vert\sum_{k=1}^N
N^{-1}\biggl(\phi-\int_Q\phi d\mu_k\biggr)\biggr\vert^2d\mu_j\le \sum_{j=1}^N\sum_{k=1}^N
N^{-1}\int_Q\biggl\vert\biggl(\phi-\int_Q\phi d\mu_k\biggr)\biggr\vert^2d\mu_j\\[0.4em]
&=\sum_{j=1}^N\sum_{k=1}^N
N^{-1}
\int_Q\biggl\vert\phi-\int_Q\phi d\mu_j+
\biggl(\int_Q\phi d\mu_j-\int_Q\phi d\mu_k\biggr)
\biggr\vert^2d\mu_j\\[0.4em]
&\le 2\sum_{j=1}^N\sum_{k=1}^N
N^{-1}\biggl\{\int_Q\biggl\vert\phi-\int_Q\phi d\mu_j
\biggr\vert^2d\mu_j+2N^2\Vert\widetilde{\nabla}(e_\kappa\phi)\bigr\Vert^2_{L^2(Q,d\mu)}\biggr\}\\[0.4em]
&=2\sum_{j=1}^N\biggl\{\int_Q\biggl\vert\phi-\int_Q\phi d\mu_j
\biggr\vert^2d\mu_j+2N^2\Vert\widetilde{\nabla}(e_\kappa\phi)\bigr\Vert^2_{L^2(Q,d\mu)}\biggr\}\\[0.4em]
&\le2\sum_{j=1}^N\biggl\{\pi^{-2}\Vert\widetilde{\nabla}(e_\kappa\phi)\bigr\Vert^2_{L^2(Q,d\mu_j)}+2N^2\Vert\widetilde{\nabla}(e_\kappa\phi)\bigr\Vert^2_{L^2(Q,d\mu)}\biggr\}\\[0.4em]
&\le 2\bigl(\pi^{-2}+2N^3\bigr)\Vert\widetilde{\nabla}(e_\kappa\phi)\bigr\Vert^2_{L^2(Q,d\mu)}.
\end{align*}
Combining the above bound
and the result of Proposition \ref{curl_grad_prop} 
applied for each $j=1,\dots, N,$ we obtain 
\begin{equation}
\int_Q\biggl\vert\phi-\int_Q\phi\biggr\vert^2d\mu\le C_{\rm P}\bigl\|\curl(e_\kappa \phi)\bigr\|^2_{L^2(Q,d\mu)},
\label{phi_fin}
\end{equation}
with  
\begin{equation}
C_{\rm P}=2\bigl(\pi^{-2}+2N^3\bigr),
\label{CP}
\end{equation}
which we note depends on $N$ only.

According to the result of Section \ref{smooth_approx}, the pair $({\mathfrak u}, \curl(e_\kappa{\mathfrak u}))$ is approximated by functions $\phi_n\in[C^\infty_\#]^3$ satisfying the conditions of Proposition \ref{curl_grad_prop}, {\it i.e.} such that ({\it cf.} (\ref{weak_div_cond}))
\[
\int_Qe_\kappa\phi_n\cdot {\nabla(e_\kappa\psi)}\,d\mu=0\qquad \forall\psi\in C_\#^\infty
\]
and $\curl(e_\kappa\phi_n)$ is pointwise orthogonal to ${\rm supp}(\mu),$
where the approximation is understood in the sense that ({\it cf.} (\ref{phin_approx}))
\[
\bigl(e_\kappa \phi_n, \curl(e_\kappa\phi_n)\bigr)\stackrel{n\to\infty}{\longrightarrow}\bigl({\mathfrak u}, \curl(e_\kappa{\mathfrak u})\bigr)\ \ {\rm in}\ \ L^2(Q, d\mu) \oplus L^2(Q, d\mu).
\]


Writing the bound (\ref{phi_fin}) with $\phi=\phi_n,$ where $\{\phi_n\}\subset[C^\infty_\#]^3$ is the approximating sequence for $u$ as described above, and passing to the limit as $n\to\infty$
yields the inequality (\ref{poincare_v}) with the constant $C_{\rm P}$ given by (\ref{CP}).

\section*{Acknowledgments}
We are grateful to Igor Vel\v{c}i\'{c} for fruitful discussions, which motivated our version of the Poincar\'{e} inequality of Section \ref{Helmholtz_section}. KC is grateful for the support of
the Engineering and Physical Sciences Research Council: Grant EP/L018802/2 ``Mathematical foundations of metamaterials: homogenisation, dissipation and operator theory''.

\end{document}